\documentclass[12pt,a4paper]{amsart}
\usepackage[foot]{amsaddr}
 
\usepackage{pdfsync}
\usepackage[stable]{footmisc}
\usepackage[utf8]{inputenc}
\usepackage{graphicx}
\usepackage{amsmath}
\usepackage{amsfonts}
\usepackage{amssymb}
\usepackage{amsthm}
\usepackage[margin=2cm]{geometry}
\usepackage[utf8]{inputenc}
\usepackage{amsmath}
\usepackage{amsfonts}
\usepackage{amssymb}
\usepackage{amsthm}
\usepackage{geometry}
\usepackage{tikz}
\usetikzlibrary{shapes.geometric}
\usepackage{mathtools}
\usepackage[dvipsnames]{xcolor}
\usepackage{url}
\usepackage{enumitem}
\usepackage{float}
\usepackage{graphicx}
\usepackage[dvipsnames]{xcolor}
\usepackage{multirow}
\usepackage[hypertexnames=false]{hyperref}
\newcommand{\RR}{\mathbb{R}}

\newcommand{\eps}{\varepsilon}

\newcommand{\dd}{\mathrm{d}}

\newcommand{\BB}{\mathcal{B}}
\newcommand{\TT}{\mathcal{T}}

\newtheorem{theorem}{Theorem}[section]

\newtheorem{proposition}[theorem]{Proposition}
\newtheorem{lemma}[theorem]{Lemma}

\theoremstyle{definition}

\newtheorem{remark}[theorem]{Remark}
\numberwithin{equation}{section}

\usepackage[backend=bibtex,style=numeric,citestyle=numeric, doi=false,isbn=false,maxbibnames=99]{biblatex}
\title{Galerkin Approximation of the Fractional Hardy Constant}

       	\author[A.~Dima]{ANDREEA DIMA}
	\address[A.~Dima]{``Simion Stoilow''  Institute of Mathematics of the Romanian Academy, 21 Calea Grivi\c{t}ei Street, 010702, Bucharest, Romania}
\email[A.~Dima]{andreeadima21\@@{}gmail.com}
    
\author[L.~I.~Ignat]{Liviu I. Ignat}
\address[L.~I.~Ignat]{Institute of Mathematics ``Simion Stoilow'' of the Romanian Academy, 21 Calea Grivitei Street, 010702 Bucharest, Romania.
    \newline\indent
 National University of Science and Technology Politehnica Bucharest, 313 Splaiul Independen\c tei, 060042 Bucharest, Romania.    \newline\indent
 Academy of
Romanian Scientists, Ilfov Street, no. 3, Bucharest, Romania.}
\email[L.~I.~Ignat]{liviu.ignat\@@{}gmail.com}
        \subjclass[2020]{65N30, 46E35}
\keywords{Fractional Hardy Inequality, Approximation and Stability, Finite Element Method, Logarithmic Improvement}

\begin{document}
\begin{abstract}
We establish sharp estimates for the discrete optimal constant of the fractional Hardy Inequality in dimension $N\geq 1$, with fractional exponent $s\in \left(0,\min\left\{1,\frac{N}{2}\right\}\right)$. The convergence rates that we establish take place for the Galerkin approximation with piecewise linear elements, when the computations are carried out in a bounded, convex and smooth domain containing the origin, for which we employ a quasi-uniform and regular mesh.
\end{abstract}

\maketitle

\section{Introduction}

 The approximation of optimal constants in functional inequalities by finite-dimensional
numerical schemes is a natural   companion problem to the analysis of the inequalities themselves: it quantifies how much of the sharp continuous theory survives discretization,
and the rate of convergence often reflects deep structural features of the extremizing
sequence -- concentration, non-attainment, or symmetry -- that are otherwise invisible at
the continuous level. This program has been carried out for Sobolev's inequality \cite{pratelli,ignat2025optimalconvergenceratesfinite},  fractional   Sobolev's inequality\cite{dima2026galerkin}, and  for Hardy's inequality \cite{ignat2024,Ignat2025}; the present paper extends it to the fractional Hardy
inequality.

Hardy's Inequality, which is central in analysis and mathematical physics states that
\begin{equation*}
\int_{\RR^N} |Du(x)|^p \, \dd x\geq \left(\frac{|N-p|}{p}\right)^p \int_{\RR^N} \frac{|u(x)|^p}{|x|^p} \, \dd x,
\end{equation*}
where $u\in C_0^{\infty}(\RR^N)$ if $1\leq p<N$, while if $p>N$, $u\in C_0^{\infty}(\RR^N\setminus \{0\})$. The constant $\lambda_{N,p}\coloneqq \left(\frac{|N-p|}{p}\right)^p$ is sharp and, for $p>1$, it is not attained in $\dot{W}^{1,p}(\RR^N)$, respectively in $\dot{W}^{1,p}(\RR^N\setminus \{0\})$, while if $p=1$ it is attained if $u$ is a symmetric decreasing function.

In \cite{ignat2024,Ignat2025} the authors used the Galerkin approximation with piecewise linear elements in order to determine how far is the discrete optimal constant to the continuous one, in the case $p=2$ and $N\neq 2$. They obtained a logarithmic order of convergence, namely that $\lambda_{N,2,h}-\lambda_{N,2}\sim \frac{1}{|\log h|^2}$. 

%Their methods differ from the ones developed for the Sobolev's Inequality and the lower bound is established using a logarithmic-improvement of Hardy's Inequality from \cite{peral2021elliptic},
%\begin{equation*}
%\int_{\Omega} |Du|^2 \, \dd x-\lambda_{N,2}\int_{\Omega} \frac{u^2}{|x|^2} \, \dd x\geq C(N,\Omega) \int_{\Omega} |Du|^2 \left(\log\left(\frac{1}{|x|}\right)\right)^{-2} \, \dd x,\ u\in C^\infty_c(\Omega),
%\end{equation*}
%where $0\in \Omega\subset\overline{\Omega}\subset B_1(0)$ is a bounded domain and $C(N,\Omega)$ is a positive constant. The upper bound is proved by considering an approximation of the function $\tilde{u}(x) \coloneqq |x|^{-\frac{N}{2}+1} \left(\log\left(\frac{1}{|x|}\right)\right)^\alpha$, for $\alpha>\frac{1}{2}$ and project it on the finite-dimensional space $V_h$.

The fractional analogue of Hardy's inequality reads
\[
\int_{\RR^N} \int_{\RR^N} \frac{|u(x)-u(y)|^p}{|x-y|^{N+ps}} \, \dd x\, \dd y \geq {\lambda}_{N,s}(p) \int_{\RR^N} \frac{|u(x)|^p}{|x|^{ps}} \, \dd x,
\]
  where $s\in (0,1)$, $u\in \dot{W}^{s,p}(\RR^N)$ if  $1\leq p<\frac{N}{s}$ and $u\in\dot{W}^{s,p}(\RR^N\setminus\{0\})$ if $p>\frac{N}{s}$. 
The best constant in the general case has been obtained  in \cite{FRANK20083407}; even for $p=2$ is known from the related work  in \cite{Herbst1977}.

In this paper we extend these results in \cite{ignat2024,Ignat2025} by obtaining the order of approximation of the fractional Hardy constant in the case $p=2$. 
In this particular case $$\lambda_{N,s}:=\inf_{u\in \dot{H}^s(\RR^N),\ u\neq 0} \frac{[u]^2_{\dot{H}^s(\RR^N)}}{\|u|x|^{-s}\|^2_{L^2(\RR^N)}}=\frac{2\pi^{\frac{N}{2}} \Gamma^2\left(\frac{N+2s}{4}\right)\Gamma\left(\frac{N+2s}{2}\right)}{\Gamma^2\left(\frac{N-2s}{4}\right) |\Gamma(-s)|}.$$
where we denoted 
\begin{equation*}
[u]^2_{\dot{H}^s(\RR^N)} \coloneqq  \int_{\RR^N} \int_{\RR^N} \frac{|u(x)-u(y)|^2}{|x-y|^{N+2s}} \, \dd x\, \dd y.
\end{equation*}
The above infimum is not achieved and a so-called "pseudo-minimizer" is the function $$u_1(x)=|x|^{-\frac{N-2s}{2}}$$
(see \cite[Proposition 1.2]{chenweth2021} for $N\geq 2$).

Given a bounded domain $\Omega\subset\RR^N$ containing the origin, it is known that $\lambda_{N,s}(\Omega)=\lambda_{N,s}$, where $$\lambda_{N,s}(\Omega)\coloneqq \inf_{u\in \dot{H}^s_0(\Omega),u\neq 0} \frac{[u]^2_{\dot{H}^s(\RR^N)}}{\|u|x|^{-s}\|^2_{L^2(\Omega)}},$$
and $\dot{H}^s_0(\Omega)$ is the homogeneous fractional Sobolev space, defined as the completion of the functions in $C^{\infty}_0(\Omega)$, extended with zero outside $\Omega$, under the seminorm $[u]^2_{\dot{H}^s(\RR^N)}$ (see \cite{Tzirakis23}). The above infimum is also not achieved.

We can approximate the fractional Hardy constant $\lambda_{N,s}$ by the corresponding finite-dimensional minimization problem in the space $V_h$, consisting of piecewise linear functions associated to a finite element mesh in a bounded, convex and smooth domain $\Omega$ containing the origin and extended with 0 outside $\Omega_h$, where $\Omega_h$ is the polyhedral domain that approximates $\Omega$: $$\lambda_{N,s,h}(\Omega)=\inf_{u_h\in V_h} \frac{[u_h]^2_{\dot{H}^s(\RR^N)}}{\|u_h|x|^{-s}\|^2_{L^2(\Omega)}}.$$
The following theorem provides sharp bounds for the difference $\lambda_{N,s,h}(\Omega)-\lambda_{N,s}$, which are the same as in the classical case, i.e. $s=1$, obtained in \cite{Ignat2025} and \cite{ignat2024}:
\begin{theorem}
\label{thm:main}
Let $\Omega\subset\RR^N$ be a bounded, convex and smooth domain containing the origin. For any small enough $h>0$ it holds that: $$\lambda_{N,s,h}(\Omega)-\lambda_{N,s}\sim\frac{1}{|\log h|^2}.$$
\end{theorem}

The main ingredients to obtain a such result are the following ones. The first one is  a logarithmic improvement in the right hand side of the Hardy's inequality as in    
\cite[Theorem 5]{TZIRAKIS20164513} (see Section \ref{sec:lower-bound}). This will give us a lower bound $\gtrsim |\log h|^{-2}$ in the above theorem. The second one is a ground state representation as in \cite[Prop.~4.1]{frankJAMS2008} together with a well-chosen competitor $u_\eps$ for the pseudo-minimizer $u_1$ such that 
\[
\frac{[u_\eps]_{\dot{H}^s(\RR^N)}^2}{\|u_\eps |x|^{-s}\|^2_{L^2(\RR^N)}}-\lambda_{N,s}\simeq |\log \eps|^{-2}
.\]
These ingredients together with the classical properties of the piecewise linear finite element space approximations will give us the main result of this paper.

The paper is organized as follows: in Section \ref{sec:discrete-framework} we present the discrete framework, while in Sections \ref{sec:lower-bound} and \ref{sec:upper-bound} we derive the lower and upper bounds respectively, completing the proofs of the main theorem. In Section \ref{sec:future-research} we present what can be done for some $L^p$-versions of the fractional Hardy constant and what difficulties one faces. Finally, the Appendix collects the proofs of several technical results that play an important role in deriving the optimal convergence rate from Theorem \ref{thm:main}.

\section{The discrete framework}
\label{sec:discrete-framework}
Let $\Omega\subset\RR^N$ be a bounded, convex and smooth domain and $\mathcal{T}_h$ be a regular and quasi-uniform triangulation of $\Omega$ (the set of intervals when $N=1$/triangles when $N=2$/tetrahedrons when $N=3$ etc.) i.e. if we denote by $h_T$ the diameter of the triangle $T\in \mathcal{T}_h$, by $\rho_T$ the diameter of the largest ball contained in T and by $h=\max_{T\in \mathcal{T}_h} h_T$, then there exists a positive constant $\sigma$ (independent of $h$) such that $$\frac{h_T}{\rho_T}\leq \sigma$$ for any $T\in\mathcal{T}_h$ and for any $h>0$, respectively $$\inf_{h>0} \frac{\min_{T\in \mathcal{T}_h} h_T}{\max_{T\in\mathcal{T}_h} h_T} \coloneqq \rho>0.$$ Let $\Omega_h$ be the union of the triangles in $\mathcal{T}_h$ (such that all the nodes of $\partial \Omega_h$ lie on $\partial \Omega$) and
\[V_h\coloneqq\left\{f\in C(\bar{\Omega}): f \text{ is affine on each }T\in \TT_h\text{ and } f\vert_{\bar\Omega\setminus \Omega_h}\equiv 0\right\} \subset H_0^1(\Omega).\] 

We introduce the following notation: for two expressions $E$ and $F$, we write $E\lesssim F$ if there exists a constant $C>0$ which depends only on the dimension $N$, the fractional exponent $s\in (0,1)$, and the constants $\sigma$ and $\rho$ such that $E\leq C\, F$. We also write $E\sim F$ provided that both $E\lesssim F$ and $F\lesssim E$ hold true. If the constants involved also depend on some other parameters e.g. $p$, $q$, then we write $\lesssim_p$, $\lesssim_{p,q}$, $\sim_q$ etc.

We also denote by $B_R(0)$ the open ball of radius $R>0$ in $\RR^N$ centered at the origin and by $\BB$ the unit ball in $\RR^N$ centered at the origin.

\section{Proof of the lower bound in Theorem \ref{thm:main}}
\label{sec:lower-bound}
The main tool that we employ to prove the lower bound in Theorem \ref{thm:main} is the following logarithmic-improvement of the fractional Hardy Inequality on domains:
\begin{proposition}{\cite[Theorem 5]{TZIRAKIS20164513}}\label{prop:log-improvement}
Let $N\geq 1$, $s\in \left(0,\min\left\{1,\frac{N}{2}\right\}\right)$ and $\Omega\subset\RR^N$ be a bounded domain containing the origin. There exists a constant $C=C_{N,s}$ such that the following inequality holds for any function $u\in\dot{H}^s_0(\Omega)$:
\begin{equation}
[u]^2_{\dot{H}^s(\RR^N)}\geq \lambda_{N,s} \int_{\RR^N} \frac{u^2(x)}{|x|^{2s}} \, \dd x+C\int_{\Omega} \frac{u^2(x)}{|x|^{2s}} \frac{1}{\log^2\left(\frac{eD}{|x|}\right)} \, \dd x,
\end{equation}
where $D=\sup_{x\in\Omega} |x|$.\\
\end{proposition}
\begin{proof} (of the lower bound in Theorem \ref{thm:main}) Without loss of generality, we consider $0\in\Omega\subset\overline{\Omega}\subset \BB$. For $h<<1$, let $\TT_h$ be a regular and quasi-uniform triangulation of $\Omega$ and $\Omega_h\coloneqq\cup_{T\in\TT_h} T$. Let also $u_h\in V_h$ be a minimizer corresponding to $\lambda_{N,s,h}(\Omega)$ such that $\|u_h|x|^{-s}\|_{L^2(\Omega_h)}=1$.\\
Then 
\begin{equation}
\label{eq:lowerbound1}
\begin{aligned}
\lambda_{N,s,h}(\Omega)-\lambda_{N,s}&=[u_h]^2_{\dot{H}^s(\RR^N)}-\lambda_{N,s}\\
&\geq C_{N,s} \int_{\Omega_h} \frac{u_h^2(x)}{|x|^{2s}} \frac{1}{\log^2\left(\frac{eD}{|x|}\right)} \, \dd x\\
&\geq C_{N,s}\int_{\Omega_h} \frac{u_h^2(x)}{|x|^{2s}} \frac{1}{\log^2\left(\frac{e}{|x|}\right)} \, \dd x,
\end{aligned}
\end{equation}
since $D\leq 1$. Next,
\begin{equation*}
2\int_{\Omega_h} \frac{u_h^2(x)}{|x|^{2s}} \frac{1}{\log^2\left(\frac{e}{|x|}\right)} \, \dd x\geq \frac{1}{\log^2\left(\frac{e}{h}\right)}\int_{\Omega_h\setminus B_h(0)} \frac{u_h^2(x)}{|x|^{2s}} \, \dd x+\int_{B_{2h}(0)} \frac{u_h^2(x)}{|x|^{2s}\log^2\left(\frac{e}{|x|}\right)} \, \dd x.
\end{equation*}
Since the function $r\to r^s\log\left(\frac{e}{r}\right)$ is increasing near $r\simeq 0$, we get that
\begin{equation*}
\int_{B_{2h}(0)} \frac{u_h^2(x)}{|x|^{2s}\log^2\left(\frac{e}{|x|}\right)} \, \dd x\geq \frac{1}{(2h)^{2s} \log^2\left(\frac{e}{2h}\right)}\int_{B_{2h}(0)} u_h^2(x) \, \dd x.
\end{equation*}
Consequently, 
\begin{equation}
\label{eq:lowerbound4}
\begin{split}
\lambda_{N,s,h}(\Omega)-\lambda_{N,s}&\geq \frac{C_{N,s}}{2} \left(\frac{1}{\log^2\left(\frac{e}{h}\right)}\int_{\Omega_h\setminus B_h(0)} \frac{u_h^2(x)}{|x|^{2s}} \, \dd x + \frac{1}{(2h)^{2s} \log^2\left(\frac{e}{2h}\right)}\int_{B_{2h}(0)} u_h^2(x) \, \dd x\right)\\
&\geq \frac{C_{N,s}}{2^{2s+1}} \frac{1}{\log^2\left(\frac{e}{h}\right)} \left(\int_{\Omega_h\setminus B_h(0)} \frac{u_h^2(x)}{|x|^{2s}} \, \dd x + \frac{1}{h^{2s}}\int_{B_{2h}(0)} u_h^2(x) \, \dd x \right).
\end{split}
\end{equation}
If $s\neq \frac{1}{2}$ we use Lemma \ref{lem:hardy-type-inequality} with $r=h$, $\alpha=\frac{N-2s}{2} h^{2s-2}$ and $\beta=1$ and we get that
\begin{equation*}
\begin{aligned}
\frac{N-2s}{2} \int_{B_{h}(0)}\frac{u_h^2(x)}{|x|^{2s}} \, \dd x&\leq (N+1) h^{-2s} \int_{B_h(0)} u_h^2(x) \, \dd x\\
&+ h^{2-2s}\left(1+\frac{2}{N-2s}\right) \int_{B_{h}(0)} |D u_h(x)|^2 \, \dd x. 
\end{aligned}
\end{equation*} 
The inverse estimate from \cite[Theorem 4.5.11]{BrennerScott} tells us that
\begin{equation}
\label{eq:inverse-estimate}
\sum_{T\in \TT_h,\ T\subset B_{2h}(0)} \int_{T} |D u_h(x)|^2 \, \dd x\leq \frac{\tilde{C}}{h^2} \sum_{T\in \TT_h,\ T\subset B_{2h}(0)} \int_{T} u_h^2(x) \, \dd x
\end{equation}
for some constant $\tilde{C}=\tilde{C}_{\sigma,\rho}>0$.
Thus, we get:
\begin{equation*}
\begin{aligned}
\frac{N-2s}{2} \int_{B_{h}(0)}\frac{u_h^2(x)}{|x|^{2s}} \, \dd x&\leq (N+1) h^{-2s} \int_{B_h(0)} u_h^2(x) \, \dd x\\
&+\tilde{C} \left(1+\frac{2}{N-2s}\right) h^{-2s} \int_{B_{2h}(0)} u_h^2(x) \, \dd x.
\end{aligned}
\end{equation*}
This implies that
\begin{equation}
\label{eq:lower-bound-first-case}
h^{-2s}\int_{B_{2h}(0)} u_h^2(x)\, \dd x\gtrsim \int_{B_h(0)} \frac{u_h^2(x)}{|x|^{2s}} \, \dd x.
\end{equation}
Plugging \eqref{eq:lower-bound-first-case} into \eqref{eq:lowerbound4} and using the fact that $\|u_h |x|^{-s}\|_{L^2(\Omega_h)}=1$ we get the desired conclusion.

If $s=\frac{1}{2}$ we use Lemma \ref{lem:hardy-type-inequality-2} with $r=h$ and $u=u_h$ to obtain:
\begin{equation*}
(N-1) \int_{B_h(0)} \frac{u_h^2(x)}{|x|} \, \dd x\leq (N+2) h^{-1} \int_{B_h(0)} u_h^2(x) \, \dd x+2h\int_{B_h(0)} |Du_h(x)|^2 \, \dd x.
\end{equation*}
Using again the inverse estimate from \eqref{eq:inverse-estimate}, we get
\begin{equation}
\label{eq:lower-bound-second-case}
h^{-1} \int_{B_{2h}(0)} u_h^2(x) \, \dd x\gtrsim \int_{B_h(0)} \frac{u_h^2(x)}{|x|} \, \dd x.
\end{equation}
The proof in this case finishes by plugging \eqref{eq:lower-bound-second-case} into \eqref{eq:lowerbound4}.
The proof of the lower bound in Theorem \ref{thm:main} is now completed.
\end{proof}
\section{Proof of the upper bound in Theorem \ref{thm:main}}
\label{sec:upper-bound}
It is sufficient to prove the upper bound in Theorem \ref{thm:main} when $\Omega=\BB$. Indeed, given that the constant $\lambda_{N,s}$ is independent of the domain under consideration, the upper bound holds in the general case by comparison.
We will use an approximating sequence $u_\eps$ of $$\tilde{u}_1(x)=|x|^{-\frac{N-2s}{2}} \log^{\alpha}\left(\frac{1}{|x|}\right),$$
where $\alpha\geq 1$.
We consider the following cut-off function:
\begin{equation*}
\eta_\eps(x)=
\begin{cases}
&0, \quad \text{if}\ |x|\leq\eps^2,\\
&\xi\left(\frac{\log\left(\frac{|x|}{\eps^2}\right)}{\log\left(\frac{1}{\eps}\right)}\right), \quad \text{if}\ |x|\in (\eps^2,\eps),\\
&1, \quad \text{if}\ |x|\geq \eps,\\
\end{cases}
\end{equation*}
where $\xi:[0,1]\rightarrow [0,1]$ is a smooth function such that for some $\mu\in \left(0,\frac{1}{2}\right)$, $\xi=0$ on $[0,\mu]$ and $\xi=1$ on $[1-\mu,1]$.\\
The function $\eta_\eps$ satisfies the following properties:
\begin{enumerate}[label={\arabic*.}]
\item $\eta_\eps(x)=0,\quad \text{if}\ |x|\leq \eps^{2-\mu}$,\\
\item $\eta_\eps(x)=1, \quad \text{if}\ |x|\geq \eps^{1+\mu},$\\
\item $|\eta'_\eps(r)|\lesssim \frac{1}{r |\log\eps|}$ if $r\in (\eps^2,\eps)$ and vanishes otherwise, which implies that $|D\eta_\eps(x)|\leq \frac{1}{|x| |\log\eps|},\quad \text{if}\ |x|\in (\eps^2,\eps)$ and vanishes otherwise,\\
\item $|\eta''_\eps(r)|\lesssim \frac{1}{r^2 |\log\eps|}$ if $r\in (\eps^2,\eps)$ and vanishes otherwise, which implies that $|D^2\eta_\eps(x)|\leq \frac{1}{|x|^2 |\log\eps|},\quad \text{if}\ |x|\in (\eps^2,\eps)$ and vanishes otherwise.\\
\end{enumerate}
With this function we introduce $$u_\eps(x)=\tilde{u}_1(x) \eta_\eps(x)\psi(|x|)\in C^{\infty}_c(\BB),$$
where $\psi\in C^{\infty}_c(\RR)$ such that $\psi(r)=1$ if $|r|\leq \frac{1}{4}$ and $\psi(r)=0$ if $|r|\geq\frac{1}{2}$.\\

\vspace{0.5cm}

The following Lemma provides quantitative estimates for the sequence $u_\eps$, viewed as a minimizing sequence for the fractional Hardy constant: 
\begin{lemma}
\label{lem:minimizing-sequence}
Let $\alpha\geq 1$. The family of functions $u_\eps$ satisfies the following estimates:
\begin{equation}
\label{eq:aeps1}
A_\eps\coloneqq\int_{\RR^N} \int_{\RR^N} \frac{|u_\eps(x)-u_\eps(y)|^2}{|x-y|^{N+2s}}\, \dd x\, \dd y-\lambda_{N,s} \int_{\BB} \frac{u_\eps^2(x)}{|x|^{2s}}\, \dd x \lesssim |\log\eps|^{2\alpha-1}
\end{equation} 
and 
\begin{equation}
\label{eq:ceps2}
B_\eps\coloneqq\int_{\BB} \frac{u_\eps^2(x)}{|x|^{2s}} \sim |\log\eps|^{2\alpha+1}.
\end{equation}
\end{lemma}
\begin{remark}
As a consequence, we get that 
\begin{equation}
\frac{A_\eps}{B_\eps} \lesssim \frac{1}{|\log\eps|^2}.
\end{equation}
\end{remark}
\begin{proof}
Note that $u_\eps(x)=0$ if $|x|\geq\frac{1}{2}$ or if $|x|\leq \eps^{2-\mu}$. Then, since $|\eta_\eps(x)|,|\psi(|x|)|\leq 1$ for any $x\in \BB$, we get that $$B_\eps\lesssim \int_{\eps^{2-\mu}}^{\frac{1}{2}} \frac{|\log r|^{2\alpha}}{r}\, \dd r\lesssim |\log\eps|^{2\alpha+1}$$
and $$B_\eps\gtrsim \int_{\eps}^{\frac{1}{4}} \frac{|\log r|^{2\alpha}}{r}\, \dd r\gtrsim |\log\eps|^{2\alpha+1}.$$
Then the estimate in \eqref{eq:ceps2} holds.\\
To prove the inequality $\eqref{eq:aeps1}$ we use the ground state representation proved in \cite[Prop.~4.1]{frankJAMS2008} and we get that 

\begin{equation}
\label{eq:representation-for-ueps}
A_\eps= \int_{\RR^N}\int_{\RR^N} \frac{|\theta_\eps(x)-\theta_\eps(y)|^2}{|x-y|^{N+2s}} \frac{\, \dd x\, \dd y}{|x|^{\frac{N-2s}{2}} |y|^{\frac{N-2s}{2}}},
\end{equation}
where $$\theta_\eps(x)=u_\eps(x) |x|^{\frac{N-2s}{2}}=\log^{\alpha}\left(\frac{1}{|x|}\right) \eta_\eps(x) \psi(|x|).$$
Note that $\theta_\eps(x)=0$ if $|x|\geq \frac{1}{2}$ or if $|x|\leq \eps^{2-\mu}$. Also, since $$D\theta_\eps(x)=D\eta_\eps(x) \log^{\alpha}\left(\frac{1}{|x|}\right) \psi(|x|)-\alpha\log^{\alpha-1}\left(\frac{1}{|x|}\right) \eta_{\eps}(x) \frac{x}{|x|^2} \psi(|x|)+\log^{\alpha}\left(\frac{1}{|x|}\right)\eta_\eps(x) \psi'(|x|) \frac{x}{|x|}$$ we get:
\begin{equation}
%\label{eq:grad-theta-eps}
|D\theta_\eps(x)|\lesssim 
\begin{cases}
& \frac{1}{|x|} |\log\eps|^{\alpha-1}, \quad \text{if}\ |x|\in (\eps^2,\eps),\\
&\frac{1}{|x|} \log^{\alpha-1}\left(\frac{1}{|x|}\right), \quad \text{if}\ |x|\in \left[\eps,\frac{1}{4}\right],\\
&1, \quad \text{if}\ |x|\in \left(\frac{1}{4},\frac{1}{2}\right).
\end{cases}
\end{equation}
Thus, 
\begin{equation}
\label{eq:grad-theta-eps}
|D\theta_\eps(x)|\lesssim \frac{1}{|x|} \log^{\alpha-1}\left(\frac{1}{|x|}\right)\lesssim  \frac{1}{|x|} |\log\eps|^{\alpha-1},
\end{equation}
for any $x\in \RR^N$ with $|x|\in \left(\eps^2,\frac{1}{2}\right)$.

Let $x,y\in\RR^N$ with $\eps^2<|y|\leq |x|<\frac{1}{2}$ and $p=\frac{4}{3}$. Using H\"{o}lder's Inequality and \eqref{eq:grad-theta-eps}, we can bound the difference $|\theta_\eps(x)-\theta_\eps(y)|$ as follows:
\begin{equation*}
\begin{aligned}
|\theta_\eps(x)-\theta_\eps(y)|&\leq\|\theta_\eps'\|_{L^p(|y|,|x|)} \|1\|_{L^{\frac{p}{p-1}}(|y|,|x|)}\\
&\lesssim |x|^{1-\frac{1}{p}} \left(\int_{|y|}^{|x|} \frac{1}{r^p}  |\log \eps|^{(\alpha-1)p} \, \dd r\right)^{\frac{1}{p}}\\
&\lesssim |\log \eps|^{\alpha-1} \left(\frac{|x|}{|y|}\right)^{1-\frac{1}{p}}.
\end{aligned}
\end{equation*}
Therefore,
\begin{equation}
\label{eq:diff-1}
|\theta_\eps(x)-\theta_\eps(y)|^2\lesssim |\log\eps|^{2\alpha-2} \left(\frac{|x|}{|y|}\right)^{\frac{1}{2}}.
\end{equation}

We split the integral in equality \eqref{eq:representation-for-ueps} in three pieces:
\begin{equation}
\label{eq:J1}
J_1\coloneqq \int_{B_{\frac{1}{2}}(0)\setminus B_{\eps^{2-\mu}}(0)} \int_{B_{\frac{1}{2}}(0)\setminus B_{\eps^{2-\mu}}(0)} \frac{|\theta_\eps(x)-\theta_\eps(y)|^2}{|x-y|^{N+2s}} \frac{\, \dd x\, \dd y}{|x|^{\frac{N-2s}{2}} |y|^{\frac{N-2s}{2}}},
\end{equation} 
\begin{equation}
\label{eq:J2}
J_2\coloneqq 2\int_{B_{\frac{1}{2}}(0)\setminus B_{\eps^{2-\mu}}(0)} \int_{B_{\eps^{2-\mu}}(0)} \frac{|\theta_\eps(x)|^2}{|x-y|^{N+2s}} \frac{\, \dd x\, \dd y}{|x|^{\frac{N-2s}{2}} |y|^{\frac{N-2s}{2}}}, 
\end{equation}
and 
\begin{equation}
\label{eq:J3}
J_3\coloneqq 2\int_{B^{c}_{\frac{1}{2}}(0)} \int_{B_{\frac{1}{2}}(0)\setminus B_{\eps^{2-\mu}}(0)} \frac{|\theta_\eps(y)|^2}{|x-y|^{N+2s}} \frac{\, \dd y\, \dd x}{|x|^{\frac{N-2s}{2}} |y|^{\frac{N-2s}{2}}}.
\end{equation}
We claim that 
\begin{equation}
\label{eq:estimateJ1}
J_1\lesssim |\log \eps|^{2\alpha-1},
\end{equation}
\begin{equation}
\label{eq:estimateJ2}
J_2\lesssim |\log \eps|^{2\alpha-2},
\end{equation}
and 
\begin{equation}
\label{eq:estimateJ3}
J_3\lesssim 1.
\end{equation}
These three estimates and identity \eqref{eq:representation-for-ueps} prove \eqref{eq:aeps1}.

We now prove the above claims.

\textit{Step I. Proof of \eqref{eq:estimateJ1}}
Without loss of generality we suppose that $|x|\geq |y|$. We split $J_1$ in two pieces:
\begin{equation}
\label{eq:J11}
J_{1,1}\coloneqq \int_{B_{\frac{1}{2}}(0)\setminus B_{\eps^{2-\mu}}(0)} \int_{B_{\frac{1}{2}}(0)\setminus B_{\eps^{2-\mu}}(0), |x-y|\leq \frac{|x|}{2},\ |x|\geq |y|} \frac{|\theta_\eps(x)-\theta_\eps(y)|^2}{|x-y|^{N+2s}} \frac{\, \dd x\, \dd y}{|x|^{\frac{N-2s}{2}} |y|^{\frac{N-2s}{2}}},
\end{equation}
and
\begin{equation}
\label{eq:J12}
J_{1,2}\coloneqq \int_{B_{\frac{1}{2}}(0)\setminus B_{\eps^{2-\mu}}(0)} \int_{B_{\frac{1}{2}}(0)\setminus B_{\eps^{2-\mu}}(0), |x-y|>\frac{|x|}{2},\ |x|\geq |y|} \frac{|\theta_\eps(x)-\theta_\eps(y)|^2}{|x-y|^{N+2s}} \frac{\, \dd x\, \dd y}{|x|^{\frac{N-2s}{2}} |y|^{\frac{N-2s}{2}}}.
\end{equation} 
We estimate $J_{1,1}$ as follows:
\begin{equation*}
J_{1,1}\lesssim \int_{B_{\frac{1}{2}}(0)\setminus B_{\eps^{2-\mu}}(0)} \int_{B_{\frac{1}{2}}(0)\setminus B_{\eps^{2-\mu}}(0), |x-y|\leq\frac{|x|}{2},\ |x|\geq |y|} \frac{|D\theta_\eps(z)|^2}{|x-y|^{N+2s-2}} \frac{\, \dd x\, \dd y}{|x|^{\frac{N-2s}{2}} |y|^{\frac{N-2s}{2}}},
\end{equation*}  
where $z=tx+(1-t)y$ for some $t\in [0,1]$. Observe that, since $|x-y|\leq \frac{|x|}{2}$, we get that $|x|-|y| \leq |x-y| \leq \frac{|x|}{2}$, which implies $\frac{|x|}{2}\leq |y|\leq |x|$ and then $\eps^2<\frac{|x|}{2}\leq |x|-(1-t)|x-y|\leq |z|\leq |x|<\frac{1}{2}$. Using \eqref{eq:grad-theta-eps}, we get that $$|D\theta_\eps(z)|\lesssim |\log\eps|^{\alpha-1} \frac{1}{|z|}\lesssim |\log\eps|^{\alpha-1} \frac{1}{|x|}$$ and then,
\begin{equation}
\label{eq:estimate-J11}
\begin{aligned}
J_{1,1} &\lesssim |\log\eps|^{2\alpha-2}  \int_{B_{\frac{1}{2}}(0)\setminus B_{\eps^{2-\mu}}(0)} \int_{B_{\frac{1}{2}}(0)\setminus B_{\eps^{2-\mu}}(0)\, |x-y|\leq\frac{|x|}{2},\ |x|\geq |y|} \frac{1}{|x-y|^{N+2s-2}} \frac{1}{|x|^{N-2s+2}} \, \dd x\, \dd y\\
&\lesssim  |\log\eps|^{2\alpha-2}   \int_{B_{\frac{1}{2}}(0)\setminus B_{\eps^{2-\mu}}(0)} \frac{1}{|x|^{N-2s+2}} \int_{|w|\leq \frac{|x|}{2}} \frac{1}{|w|^{N+2s-2}} \, \dd w\, \dd x\\
&\lesssim  |\log\eps|^{2\alpha-2} \int_{B_{\frac{1}{2}}(0)\setminus B_{\eps^{2-\mu}}(0)} \frac{1}{|x|^N} \, \dd x\\
&\lesssim |\log\eps|^{2\alpha-1}.
\end{aligned}
\end{equation}
In order to estimate $J_{1,2}$ we consider two cases: if $|x|\geq 3|y|$, then $\frac{4|x|}{3}\geq |x|+|y|\geq |x-y|\geq |x|-|y|\geq \frac{2|x|}{3}$ so that $|x-y|\sim |x|$. In this case we get:
\begin{equation}
\label{eq:x-3y-1}
\begin{aligned}
&\int_{B_{\frac{1}{2}}(0)\setminus B_{\eps^{2-\mu}}(0)} \int_{B_{\frac{1}{2}}(0)\setminus B_{\eps^{2-\mu}}(0), |x-y|>\frac{|x|}{2}, |x|\geq 3|y|} \frac{|\theta_\eps(x)-\theta_\eps(y)|^2}{|x-y|^{N+2s}} \frac{\, \dd x\, \dd y}{|x|^{\frac{N-2s}{2}} |y|^{\frac{N-2s}{2}}}\\
&\lesssim \int_{B_{\frac{1}{2}}(0)\setminus B_{\eps^{2-\mu}}(0)} \int_{B_{\frac{1}{2}}(0)\setminus B_{\eps^{2-\mu}}(0), |x-y|>\frac{|x|}{2}, |x|\geq 3|y|} \frac{|\theta_\eps(x)-\theta_\eps(y)|^2}{|x|^{\frac{3N}{2}+s} |y|^{\frac{N-2s}{2}}} \, \dd x\, \dd y.
\end{aligned}
\end{equation}
Using \eqref{eq:diff-1} we get:
\begin{equation*}
\begin{aligned}
&\int_{B_{\frac{1}{2}}(0)\setminus B_{\eps^{2-\mu}}(0)} \int_{B_{\frac{1}{2}}(0)\setminus B_{\eps^{2-\mu}}(0), |x-y|>\frac{|x|}{2}, |x|\geq 3|y|} \frac{|\theta_\eps(x)-\theta_\eps(y)|^2}{|x|^{\frac{3N}{2}+s} |y|^{\frac{N-2s}{2}}} \, \dd x\, \dd y\\
& \lesssim |\log\eps|^{2\alpha-2} \int_{B_{\frac{1}{2}}(0)\setminus B_{\eps^{2-\mu}}(0)} \int_{B_{\frac{1}{2}}(0)\setminus B_{\eps^{2-\mu}}(0), |x-y|>\frac{|x|}{2}, |x|\geq 3|y|}  \frac{1}{|x|^{\frac{3N}{2}+s-\frac{1}{2}}} \frac{1}{|y|^{\frac{N-2s}{2}+\frac{1}{2}}} \, \dd x\, \dd y\\
&\lesssim |\log\eps|^{2\alpha-2} \int_{B_{\frac{1}{2}}(0)\setminus B_{\eps^{2-\mu}}(0)} \frac{1}{|x|^{\frac{3N}{2}+s-\frac{1}{2}}} \int_{\eps^{2-\mu}}^{\frac{|x|}{3}} t^{\frac{N}{2}+s-\frac{3}{2}} \, \dd t \, \dd x \\
&\lesssim |\log\eps|^{2\alpha-2}  \int_{B_{\frac{1}{2}}(0)\setminus B_{\eps^{2-\mu}}(0)} \frac{1}{|x|^N} \, \dd x\\
& \lesssim |\log\eps|^{2\alpha-1}.
\end{aligned}
\end{equation*}
Thus,
\begin{equation}
\label{eq:estimate-J121}
\int_{B_{\frac{1}{2}}(0)\setminus B_{\eps^{2-\mu}}(0)} \int_{B_{\frac{1}{2}}(0)\setminus B_{\eps^{2-\mu}}(0), |x-y|>\frac{|x|}{2}, |x|\geq 3|y|} \frac{|\theta_\eps(x)-\theta_\eps(y)|^2}{|x-y|^{N+2s}} \frac{\, \dd x\, \dd y}{|x|^{\frac{N-2s}{2}} |y|^{\frac{N-2s}{2}}}\lesssim |\log\eps|^{2\alpha-1}.
\end{equation}
If $|x|<3|y|$, since $|x-y|>\frac{|x|}{2}$, we get that $|x|\sim |y|\sim |x-y|$. We obtain:
\begin{equation*}
\begin{aligned}
&\int_{B_{\frac{1}{2}}(0)\setminus B_{\eps^{2-\mu}}(0)} \int_{B_{\frac{1}{2}}(0)\setminus B_{\eps^{2-\mu}}(0), |x-y|>\frac{|x|}{2},\ |y|\leq |x|<3|y|} \frac{|\theta_\eps(x)-\theta_\eps(y)|^2}{|x-y|^{N+2s}} \frac{\, \dd x\, \dd y}{|x|^{\frac{N-2s}{2}} |y|^{\frac{N-2s}{2}}}\\
&\lesssim\int_{B_{\frac{1}{2}}(0)\setminus B_{\eps^{2-\mu}}(0)} \int_{B_{\frac{1}{2}}(0)\setminus B_{\eps^{2-\mu}}(0), |x-y|>\frac{|x|}{2},\ |y|\leq |x|<3|y|} \frac{|\theta_\eps(x)-\theta_\eps(y)|^2}{|x|^N |y|^N} \, \dd x\, \dd y.
\end{aligned}
\end{equation*}
Using \eqref{eq:diff-1} we get the following estimate:
\begin{equation}
\label{eq:estimate-J122}
\begin{aligned}
&\int_{B_{\frac{1}{2}}(0)\setminus B_{\eps^{2-\mu}}(0)} \int_{B_{\frac{1}{2}}(0)\setminus B_{\eps^{2-\mu}}(0), |x-y|>\frac{|x|}{2}, |y|\leq |x|<3|y|} \frac{|\theta_\eps(x)-\theta_\eps(y)|^2}{|x|^N |y|^N} \, \dd x\, \dd y\\
&\lesssim |\log\eps|^{2\alpha-2} \int_{B_{\frac{1}{2}}(0)\setminus B_{\eps^{2-\mu}}(0)} \frac{1}{|x|^{N}} \int_{\frac{|x|}{3}<|y|\leq |x|} \frac{1}{|y|^N} \, \dd y\, \dd x\\
&\lesssim |\log\eps|^{2\alpha-2} \int_{B_{\frac{1}{2}}(0)\setminus B_{\eps^{2-\mu}}(0)} \frac{1}{|x|^N} \, \dd x\\
&\lesssim |\log\eps|^{2\alpha-1}.
\end{aligned}
\end{equation}
The first claim now follows from \eqref{eq:estimate-J11}, \eqref{eq:estimate-J121} and \eqref{eq:estimate-J122}.

 \textit{Step II. Proof of 
\eqref{eq:estimateJ2}.}
We split $J_2$ in two pieces:
\begin{equation*}
J_{2,1}\coloneqq 2\int_{B_{\frac{1}{2}}(0)\setminus B_{\eps^{2-\mu}}(0)} \int_{B_{\eps^{2-\mu}}(0), |x-y|\leq\frac{|y|}{2}}  \frac{|\theta_\eps(x)|^2}{|x-y|^{N+2s}} \frac{\, \dd x\, \dd y}{|x|^{\frac{N-2s}{2}} |y|^{\frac{N-2s}{2}}},
\end{equation*}
and
\begin{equation*}
J_{2,2}\coloneqq 2\int_{B_{\frac{1}{2}}(0)\setminus B_{\eps^{2-\mu}}(0)} \int_{B_{\eps^{2-\mu}}(0), |x-y|>\frac{|y|}{2}}  \frac{|\theta_\eps(x)|^2}{|x-y|^{N+2s}} \frac{\, \dd x\, \dd y}{|x|^{\frac{N-2s}{2}} |y|^{\frac{N-2s}{2}}}.
\end{equation*}
In order to estimate $J_{2,1}$ first observe that, since $|x-y|\leq\frac{|y|}{2}$, we get that $|x|\leq \frac{3|y|}{2}\leq \frac{3\eps^{2-\mu}}{2}$ and $|y|\geq \frac{2|x|}{3}\geq \frac{2\eps^{2-\mu}}{3}$. Thus, $|x|\sim|y|\sim\eps^{2-\mu}$. Then:
\begin{equation*}
\begin{split}
J_{2,1}&\lesssim \int_{B_{\frac{3\eps^{2-\mu}}{2}}(0)\setminus B_{\eps^{2-\mu}}(0)} \int_{B_{\eps^{2-\mu}}(0)\setminus B_{\frac{2\eps^{2-\mu}}{3}}(0), |x-y|\leq\frac{|y|}{2}} \frac{|\theta_\eps(x)-\theta_\eps(y)|^2}{|x-y|^{N+2s}}  \frac{\, \dd x\, \dd y}{|x|^{\frac{N-2s}{2}} |y|^{\frac{N-2s}{2}}}\\
&\lesssim \int_{B_{\frac{3\eps^{2-\mu}}{2}}(0)\setminus B_{\eps^{2-\mu}}(0)} \int_{B_{\eps^{2-\mu}}(0)\setminus B_{\frac{2\eps^{2-\mu}}{3}}(0), |x-y|\leq\frac{|y|}{2}} \frac{|D\theta_\eps(z)|^2}{|x-y|^{N+2s-2}}  \frac{\, \dd x\, \dd y}{|x|^{\frac{N-2s}{2}} |y|^{\frac{N-2s}{2}}},
\end{split}
\end{equation*}
where $z=tx+(1-t)y$ for some $t\in [0,1]$. Since $|x-y|\leq \frac{|y|}{2}$, we get that $\frac{\eps^{2-\mu}}{3}\leq\frac{|y|}{2}\leq |y|-t|x-y|\leq|z|\leq |x|+|y|\leq \frac{5\eps^{2-\mu}}{2}$. Thus, $|z|\sim \eps^{2-\mu}$. Using \eqref{eq:grad-theta-eps} we get that $$|D\theta_\eps(z)|\lesssim \frac{|\log\eps|^{\alpha-1}}{\eps^{2-\mu}}.$$ Thus,
\begin{equation}
\label{eq:estimate-on-J21}
\begin{aligned}
J_{2,1}&\lesssim \frac{|\log\eps|^{2\alpha-2}}{\eps^{2(2-\mu)}} \int_{B_{\eps^{2-\mu}}(0)\setminus B_{\frac{2\eps^{2-\mu}}{3}}(0)} \frac{1}{|y|^{N-2s}} \int_{|x-y|\leq \frac{|y|}{2}} \frac{1}{|x-y|^{N+2s-2}} \, \dd x\, \dd y\\
&\lesssim \frac{|\log\eps|^{2\alpha-2}}{\eps^{2(2-\mu)}}  \int_{B_{\eps^{2-\mu}}(0)\setminus B_{\frac{2\eps^{2-\mu}}{3}}(0)} \frac{1}{|y|^{N-2}} \, \dd y\\
&\lesssim |\log\eps|^{2\alpha-2}.
\end{aligned}
\end{equation}
We estimate $J_{2,2}$ by considering two cases:\\
\textbf{Case I.} If $|x|<2|y|$, then $|x|< 2\eps^{2-\mu}$ and $|y|> \frac{\eps^{2-\mu}}{2}$. Since $|x-y|>\frac{|y|}{2}$, we obtain $|x|\sim |y|\sim |x-y|\sim \eps^{2-\mu}$ and, thus
\begin{equation}
\label{eq:x-2y-1}
\begin{aligned}
&\int_{B_{\frac{1}{2}}(0)\setminus B_{\eps^{2-\mu}}(0)} \int_{B_{\eps^{2-\mu}}(0), |x-y|>\frac{|y|}{2}, |x|<2|y|}  \frac{|\theta_\eps(x)|^2}{|x-y|^{N+2s}} \frac{\, \dd x\, \dd y}{|x|^{\frac{N-2s}{2}} |y|^{\frac{N-2s}{2}}}\\
&\quad \lesssim \eps^{-2(2-\mu)N}\int_{B_{2\eps^{2-\mu}}(0)\setminus B_{\eps^{2-\mu}}(0)} \int_{B_{\eps^{2-\mu}}(0)\setminus B_{\frac{\eps^{2-\mu}}{2}}(0), |x-y|>\frac{|y|}{2}} |\theta_\eps(x)-\theta_\eps(y)|^2 \, \dd x\, \dd y.
\end{aligned}
\end{equation}
Using \eqref{eq:diff-1} we obtain:
\begin{equation*}
\begin{aligned}
\eps^{-2(2-\mu)N}&\int_{B_{2\eps^{2-\mu}}(0)\setminus B_{\eps^{2-\mu}}(0)} \int_{B_{\eps^{2-\mu}}(0)\setminus B_{\frac{\eps^{2-\mu}}{2}}(0), |x-y|>\frac{|y|}{2}} |\theta_\eps(x)-\theta_\eps(y)|^2 \, \dd x\, \dd y\\
&  \lesssim \eps^{-2(2-\mu)N} |\log\eps|^{2\alpha-2} \int_{B_{2\eps^{2-\mu}}(0)\setminus B_{\eps^{2-\mu}}(0)} \int_{B_{\eps^{2-\mu}}(0)\setminus B_{\frac{\eps^{2-\mu}}{2}}(0), |x-y|>\frac{|y|}{2}} 1 \, \dd x\, \dd y\\
&\lesssim |\log\eps|^{2\alpha-2}.
\end{aligned}
\end{equation*}
Thus,
\begin{equation}
\label{eq:estimateJ221}
\int_{B_{\frac{1}{2}}(0)\setminus B_{\eps^{2-\mu}}(0)} \int_{B_{\eps^{2-\mu}}(0), |x-y|>\frac{|y|}{2}, |x|<2|y|}  \frac{|\theta_\eps(x)|^2}{|x-y|^{N+2s}} \frac{\, \dd x\, \dd y}{|x|^{\frac{N-2s}{2}} |y|^{\frac{N-2s}{2}}}\lesssim |\log\eps|^{2\alpha-2}.
\end{equation}
\textbf{Case II.} If $|x|\geq 2|y|$, then $|x-y|\geq \frac{|x|}{2}$. We choose a real number $0<\delta<1-\mu$. There are three cases:\\
\textbf{Case II.1.} If $|x|\geq \eps^{2-\mu-\delta}$, then
\begin{equation}
\label{eq:estimateJ222}
\begin{aligned}
&\int_{B_{\frac{1}{2}}(0)\setminus B_{\eps^{2-\mu}}(0)} \int_{B_{\eps^{2-\mu}}(0), |x-y|>\frac{|y|}{2},\ |x|\geq 2|y|,\ |x|\geq \eps^{2-\mu-\delta}} \frac{|\theta_\eps(x)|^2}{|x-y|^{N+2s}} \frac{\, \dd x\, \dd y}{|x|^{\frac{N-2s}{2}} |y|^{\frac{N-2s}{2}}}\\
&\lesssim \int_{B_{\frac{1}{2}}(0)\setminus B_{\eps^{2-\mu-\delta}}(0)} \int_{B_{\eps^{2-\mu}}(0)}  \log^{2\alpha}\left(\frac{1}{|x|}\right) \frac{1}{|x|^{\frac{3N}{2}+s} |y|^{\frac{N-2s}{2}}}\, \dd x\, \dd y\\
&\lesssim |\log\eps|^{2\alpha} \int_{\eps^{2-\mu-\delta}}^{\frac{1}{2}} r^{-\frac{N}{2}-s-1}\, \dd r \int_{0}^{\eps^{2-\mu}} t^{\frac{N}{2}+s-1} \, \dd t\\
&\lesssim |\log\eps|^{2\alpha} \eps^{\delta\left(\frac{N}{2}+s\right)}\lesssim |\log\eps|^{2\alpha-2}.
\end{aligned}
\end{equation}
\textbf{Case II.2.} If $|x|<\eps^{2-\mu-\delta}$ and $|y|\leq\eps^2$, then
\begin{equation}
\label{eq:estimateJ223}
\begin{aligned}
&\int_{B_{\frac{1}{2}}(0)\setminus B_{\eps^{2-\mu}}(0)} \int_{B_{\eps^{2-\mu}}(0), |x-y|>\frac{|y|}{2},\ |x|\geq 2|y|,\ |x|< \eps^{2-\mu-\delta}, |y|\leq \eps^2} \frac{|\theta_\eps(x)|^2}{|x-y|^{N+2s}} \frac{\, \dd x\, \dd y}{|x|^{\frac{N-2s}{2}} |y|^{\frac{N-2s}{2}}}\\
&\lesssim |\log\eps|^{2\alpha} \int_{B_{\eps^{2-\mu-\delta}}(0)\setminus B_{\eps^{2-\mu}}(0)} \frac{1}{|x|^{\frac{3N}{2}+s}} \, \dd x\int_{B_{\eps^2}(0)} \frac{1}{|y|^{\frac{N-2s}{2}}}\, \dd y\\
&\lesssim |\log\eps|^{2\alpha} \eps^{\mu\left(\frac{N}{2}+s\right)}\lesssim |\log\eps|^{2\alpha-2}.\\
\end{aligned}
\end{equation} 
\textbf{Case II.3.} If $|x|<\eps^{2-\mu-\delta}$ and $|y|>\eps^2$, then 
\begin{equation*}
\begin{aligned}
&\int_{B_{\frac{1}{2}}(0)\setminus B_{\eps^{2-\mu}}(0)} \int_{B_{\eps^{2-\mu}}(0), |x-y|>\frac{|y|}{2},\ |x|\geq 2|y|,\ |x|< \eps^{2-\mu-\delta}, |y|>\eps^2} \frac{|\theta_\eps(x)|^2}{|x-y|^{N+2s}} \frac{\, \dd x\, \dd y}{|x|^{\frac{N-2s}{2}} |y|^{\frac{N-2s}{2}}}\\
&\lesssim \int_{B_{\eps^{2-\mu-\delta}}(0)\setminus B_{\eps^{2-\mu}}(0)}\int_{B_{\eps^{2-\mu}}(0)\setminus B_{\eps^2}(0)} \frac{|\theta_\eps(x)-\theta_\eps(y)|^2}{|x|^{\frac{3N}{2}+s} |y|^{\frac{N-2s}{2}}} \, \dd x\, \dd y.
\end{aligned}
\end{equation*}
Using \eqref{eq:diff-1} we get:
\begin{equation*}
\begin{aligned}
&\int_{B_{\eps^{2-\mu-\delta}}(0)\setminus B_{\eps^{2-\mu}}(0)}\int_{B_{\eps^{2-\mu}}(0)\setminus B_{\eps^2}(0)} \frac{|\theta_\eps(x)-\theta_\eps(y)|^2}{|x|^{\frac{3N}{2}+s} |y|^{\frac{N-2s}{2}}} \, \dd x\, \dd y\\
&\lesssim |\log\eps|^{2\alpha-2} \int_{B_{\eps^{2-\mu-\delta}}(0)\setminus B_{\eps^{2-\mu}}(0)} \frac{1}{|x|^{\frac{3N}{2}+s-\frac{1}{2}}} \, \dd x \int_{B_{\eps^{2-\mu}}(0)\setminus B_{\eps^2}(0)} \frac{1}{|y|^{\frac{N}{2}-s+\frac{1}{2}}} \, \dd y\\
&\lesssim |\log\eps|^{2\alpha-2} \int_{\eps^{2-\mu}}^{2-\mu-\delta} r^{-\frac{N}{2}-s-\frac{1}{2}} \, \dd r \int_{\eps^2}^{\eps^{2-\mu}} t^{\frac{N}{2}+s-\frac{3}{2}} \, \dd t\\
&\lesssim |\log\eps|^{2\alpha-2}.
\end{aligned}
\end{equation*}
Thus,
\begin{equation}
\label{eq:estimateJ224}
\int_{B_{\frac{1}{2}}(0)\setminus B_{\eps^{2-\mu}}(0)} \int_{B_{\eps^{2-\mu}}(0), |x-y|>\frac{|y|}{2},\ |x|\geq 2|y|,\ |x|< \eps^{2-\mu-\delta}, |y|>\eps^2} \frac{|\theta_\eps(x)|^2}{|x-y|^{N+2s}} \frac{\, \dd x\, \dd y}{|x|^{\frac{N-2s}{2}} |y|^{\frac{N-2s}{2}}}\lesssim |\log\eps|^{2\alpha-2}.
\end{equation}
The estimate of $J_2$  follows from \eqref{eq:estimate-on-J21}, \eqref{eq:estimateJ221}, \eqref{eq:estimateJ222}, \eqref{eq:estimateJ223} and \eqref{eq:estimateJ224}. 

 \textit{Step III. Proof of \eqref{eq:estimateJ3}.}
We split $J_3$ in two pieces:
\begin{equation*}
J_{3,1}\coloneqq 2\int_{B^{c}_{\frac{1}{2}}(0)} \int_{B_{\frac{1}{2}}(0)\setminus B_{\eps^{2-\mu}}(0), |y|\geq \frac{1}{|x|}} \frac{|\theta_\eps(y)|^2}{|x-y|^{N+2s}} \frac{\, \dd y\, \dd x}{|x|^{\frac{N-2s}{2}} |y|^{\frac{N-2s}{2}}},
\end{equation*}
and 
\begin{equation*}
J_{3,2}\coloneqq 2\int_{B^{c}_{\frac{1}{2}}(0)} \int_{B_{\frac{1}{2}}(0)\setminus B_{\eps^{2-\mu}}(0), |y|< \frac{1}{|x|}} \frac{|\theta_\eps(y)|^2}{|x-y|^{N+2s}} \frac{\, \dd y\, \dd x}{|x|^{\frac{N-2s}{2}} |y|^{\frac{N-2s}{2}}}.
\end{equation*}
In case of $J_{3,1}$ we get that $|x|\geq 2|y|$ which implies $|x-y|\geq \frac{|x|}{2}$ and $|x|\geq 2$. Then:
\begin{equation}
\label{eq:estimateJ31}
\begin{aligned}
J_{3,1}&\lesssim\int_{B^{c}_2(0)} \frac{1}{|x|^{\frac{3N}{2}+s}}\int_{|x|^{-1}\leq |y|\leq \frac{1}{2}} \frac{\log^{2\alpha}\left(\frac{1}{|y|}\right)}{|y|^{\frac{N-2s}{2}}} \, \dd y\, \dd x\\
&\lesssim\int_{B^{c}_2(0)} \frac{1}{|x|^{\frac{3N}{2}+s}}\int_{|x|^{-1}}^{\frac{1}{2}}\log^{2\alpha}\left(\frac{1}{r}\right) r^{\frac{N}{2}+s-1}\, \dd r\, \dd x\\
&\lesssim \int _0^{\frac 12} r^{-\frac 12}\, \dd r 
\int_{B^{c}_2(0)} \frac{ 1}{|x|^{\frac{3N}{2}+s}}\, \dd x\lesssim 1, 
\end{aligned}
\end{equation}
where we have used the fact that, since $\frac{N}{2}+s-\frac{1}{2}\geq s>0$, we get 
\begin{equation}
\label{eq:ineq-r-log-1-r}
\log^{2\alpha}\left(\frac{1}{r}\right) r^{\frac{N}{2}+s-\frac{1}{2}}\lesssim 1,\quad \forall\,r\in (0,1].
\end{equation}

In case of $J_{3,2}$ we must have $|x|<\eps^{\mu-2}$. We treat two cases: if $|x|>2|y|$, then $|x-y|\geq \frac{|x|}{2}$ and we obtain that $|x-y|\sim |x|$. Thus,
\begin{equation}
\label{eq:estimateJ321-1}
\begin{aligned}
&\int_{B^{c}_{\frac{1}{2}}(0)} \int_{B_{\frac{1}{2}}(0)\setminus B_{\eps^{2-\mu}}(0), |y|< \frac{1}{|x|},|x|>2|y|} \frac{|\theta_\eps(y)|^2}{|x-y|^{N+2s}} \frac{\, \dd y\, \dd x}{|x|^{\frac{N-2s}{2}} |y|^{\frac{N-2s}{2}}}\\
&\lesssim \int_{B^{c}_{\frac{1}{2}}(0)} \int_{B_{\frac{1}{2}}(0)\setminus B_{\eps^{2-\mu}}(0), |y|< \frac{1}{|x|},|x|>2|y|} \frac{|\theta_\eps(y)|^2}{|x|^{\frac{3N}{2}+s} |y|^{\frac{N}{2}-s}} \, \dd y \, \dd x.
\end{aligned}
\end{equation}
 
Using that $|\theta_\eps(y)|\leq \log^{\alpha}\left(\frac{1}{|y|}\right)$, we get:
\begin{equation*}
\begin{aligned}
\int_{B^{c}_{\frac{1}{2}}(0)} &\int_{B_{\frac{1}{2}}(0)\setminus B_{\eps^{2-\mu}}(0), |y|< \frac{1}{|x|},|x|>2|y|} \frac{|\theta_\eps(y)|^2}{|x|^{\frac{3N}{2}+s} |y|^{\frac{N}{2}-s}} \, \dd y \, \dd x\\
& \leq \int_{B_{\frac{1}{2}}(0)\setminus B_{\eps^{2-\mu}}(0)} \frac{\log^{2\alpha}\left(\frac{1}{|y|}\right)}{|y|^{\frac{N}{2}-s}} \, \dd y \int_{B^{c}_{\frac{1}{2}}(0)} \frac{1}{|x|^{\frac{3N}{2}+s}} \, \dd x\\
&\lesssim \int_{\eps^{2-\mu}}^{\frac{1}{2}} \log^{2\alpha}\left(\frac{1}{t}\right) t^{\frac{N}{2}+s-1} \, \dd t \int_{\frac{1}{2}}^{\infty} r^{-\frac{N}{2}-s-1} \, \dd r\lesssim \int_{\eps^{2-\mu}}^{\frac{1}{2}} t^{-\frac{1}{2}} \, \dd t\lesssim 1,
\end{aligned}
\end{equation*} 
where we have used \eqref{eq:ineq-r-log-1-r}.

Consequently, 
\begin{equation}
\label{eq:estimateJ321}
\int_{B^{c}_{\frac{1}{2}}(0)} \int_{B_{\frac{1}{2}}(0)\setminus B_{\eps^{2-\mu}}(0), |y|< \frac{1}{|x|},|x|>2|y|} \frac{|\theta_\eps(y)|^2}{|x-y|^{N+2s}} \frac{\, \dd y\, \dd x}{|x|^{\frac{N-2s}{2}} |y|^{\frac{N-2s}{2}}}\lesssim 1.
\end{equation}
If $|x|\leq 2|y|$ then $|x|\in\left[\frac{1}{2},1\right]$, $|y|\in\left[\frac{1}{4},\frac{1}{2}\right]$. If $|x-y|>\frac{|y|}{2}\geq \frac{1}{8}$, then 
\begin{equation}
\label{eq:estimateJ322}
\int_{B^{c}_{\frac{1}{2}}(0)} \int_{B_{\frac{1}{2}}(0)\setminus B_{\eps^{2-\mu}}(0), |y|< \frac{1}{|x|},|x|\leq 2|y|, |x-y|>\frac{|y|}{2}} \frac{|\theta_\eps(y)|^2}{|x-y|^{N+2s}} \frac{\, \dd y\, \dd x}{|x|^{\frac{N-2s}{2}} |y|^{\frac{N-2s}{2}}}\lesssim 1,
\end{equation}
while if $|x-y|\leq \frac{|y|}{2}$, then $|x-y|\leq \frac{1}{4}$ and we proceed as follows:
\begin{equation*}
\begin{aligned}
&\int_{B^{c}_{\frac{1}{2}}(0)} \int_{B_{\frac{1}{2}}(0)\setminus B_{\eps^{2-\mu}}(0),\ |y|< \frac{1}{|x|},\ |x|\leq 2|y|,\ |x-y|\leq\frac{|y|}{2}} \frac{|\theta_\eps(y)|^2}{|x-y|^{N+2s}} \frac{\, \dd y\, \dd x}{|x|^{\frac{N-2s}{2}} |y|^{\frac{N-2s}{2}}}\\
&\lesssim \int_{B_1(0)\setminus B_{\frac{1}{2}}(0)} \int_{B_{\frac{1}{2}}(0)\setminus B_{\frac{1}{4}}(0),\ |x-y|\leq \frac{|y|}{2}} \frac{|\theta_\eps(y)|^2}{|x-y|^{N+2s}} \, \dd y\, \dd x\\
&\lesssim  \int_{B_1(0)\setminus B_{\frac{1}{2}}(0)} \int_{B_{\frac{1}{2}}(0)\setminus B_{\frac{1}{4}}(0),\ |x-y|\leq \frac{|y|}{2}} \frac{|\theta_\eps(y)-\theta_\eps(x)|^2}{|x-y|^{N+2s}} \, \dd y\, \dd x\\
&\lesssim \int_{B_1(0)\setminus B_{\frac{1}{2}}(0)} \int_{B_{\frac{1}{2}}(0)\setminus B_{\frac{1}{4}}(0), |x-y|\leq \frac{1}{4}} \frac{|D\theta_\eps(z)|^2}{|x-y|^{N+2s-2}} \, \dd y\, \dd x,
\end{aligned}
\end{equation*}
where $z=tx+(1-t)y$ for some $t\in [0,1]$. Then $|z|\geq |y|-|x-y|\geq \frac{|y|}{2}\geq \frac{1}{8}$. Using estimates \eqref{eq:grad-theta-eps} we get that $|D\theta_\eps(z)|\lesssim 1$. Then
\begin{equation*}
\int_{B_1(0)\setminus B_{\frac{1}{2}}(0)} \int_{B_{\frac{1}{2}}(0)\setminus B_{\frac{1}{4}}(0)} \frac{|D\theta_\eps(z)|^2}{|x-y|^{N+2s-2}} \, \dd x\, \dd y\lesssim \int_{|w|<\frac{1}{4}} \frac{1}{|w|^{N+2s-2}} \, \dd w\lesssim 1.
\end{equation*}
Thus, 
\begin{equation}
\label{eq:estimateJ323}
\int_{B^{c}_{\frac{1}{2}}(0)} \int_{B_{\frac{1}{2}}(0)\setminus B_{\eps^{2-\mu}}(0),\ |x|< \frac{1}{|y|},\ |x|\leq 2|y|,\ |x-y|\leq\frac{|y|}{2}} \frac{|\theta_\eps(y)|^2}{|x-y|^{N+2s}} \frac{\, \dd y\, \dd x}{|x|^{\frac{N-2s}{2}} |y|^{\frac{N-2s}{2}}}\lesssim 1.
\end{equation}
The bound on $J_3$ follows from \eqref{eq:estimateJ31}, \eqref{eq:estimateJ321}, \eqref{eq:estimateJ322} and \eqref{eq:estimateJ323}.
 
The proof of the Lemma is now complete.
\end{proof}
We now prove the desired upper bound in Theorem \ref{thm:main}. We consider $\eps=h^{\frac{\gamma}{2}}$ for some $\gamma\in (0,1)$. We also consider the space $$\tilde{H}^s(\BB)\coloneqq \{u\in H^s(\RR^N): u=0\ \text{on } \RR^N\setminus\BB\}$$ and the Hilbert space $$(V,\| \cdot\|_V)\coloneqq (\tilde{H}^s(\BB),[\cdot]_{\dot{H}^s(\RR^N)}).$$
Let us denote by $\Pi_h : V\rightarrow V_h$ the projection on the space $V_h$ with respect to $\|\cdot\|_{V}$-norm (\cite[Section 4]{borthagaray2018}) i.e. $$[u-\Pi_h u]_{\dot{H}^s(\RR^N)}=\inf_{u_h\in V_h} [u-u_h]_{\dot{H}^s(\RR^N)}$$
and by $I_h u_\eps$ the linear interpolant of the function $u_\eps$, extended with zero outside $\BB_h$. 

By the definition of $\Pi_h$ we get that $$[\Pi_h u_\eps]_{\dot{H}^s(\RR^N)}\leq [u_\eps]_{\dot{H}^s(\RR^N)}$$ and $$[\Pi_h u_\eps-u_\eps]_{\dot{H}^s(\RR^N)} \leq [I_h u_\eps-u_\eps]_{\dot{H}^s(\RR^N)}.$$ We get:
\begin{equation}
\label{eq:lambdah-1}
\begin{aligned}
\lambda_{N,s,h}(\BB)&\leq \frac{ [\Pi_h u_\eps]^2_{\dot{H}^s(\RR^N)}}{\int_{\BB}  {|\Pi_h u_\eps(x)|^2}{|x|^{-2s}} \, \dd x}\leq \frac{ [\Pi_h u_\eps]^2_{\dot{H}^s(\RR^N)}}{\int_{\BB} {|\Pi_h u_\eps(x)|^2}{|x|^{-2s}} \, \dd x}.
\end{aligned}
\end{equation}
Using that $\frac{1}{\delta} |a-b|^2+|b|^2\geq \frac{1}{1+\delta} |a|^2$ for any real numbers $a,b,\delta$ with $\delta>0$, we obtain:
\begin{equation*}
\begin{aligned}
\int_{\BB} &\frac{|\Pi_h u_\eps(x)|^2}{|x|^{2s}} \, \dd x \\
& \geq \frac{1}{1+\delta} \int_{\BB} \frac{|u_\eps(x)|^2}{|x|^{2s}} \, \dd x-\frac{1}{\delta} \int_{\BB} \frac{|(\Pi_h u_\eps-u_\eps)(x)|^2}{|x|^{2s}} \, \dd x\\
&\geq \frac{1}{1+\delta} \int_{\BB} \frac{|u_\eps(x)|^2}{|x|^{2s}} \, \dd x-\frac{1}{\delta \lambda_{N,s}} \int_{\RR^N} \int_{\RR^N} \frac{|(\Pi_h u_\eps-u_\eps)(x)-(\Pi_h u_\eps-u_\eps)(y)|^2}{|x-y|^{N+2s}} \, \dd x\, \dd y\\
&\geq  \frac{1}{1+\delta} \int_{\BB} \frac{|u_\eps(x)|^2}{|x|^{2s}} \, \dd x-\frac{1}{\delta \lambda_{N,s}} \int_{\RR^N} \int_{\RR^N} \frac{|(I_h u_\eps-u_\eps)(x)-(I_h u_\eps-u_\eps)(y)|^2}{|x-y|^{N+2s}} \, \dd x\, \dd y.
\end{aligned}
\end{equation*}
Using the Homogeneous Gagliardo-Nirenberg Interpolation Inequality \cite[Theorem 7.45]{GLeoni2023}, we get that 
\begin{equation*}
[I_h u_\eps-u_\eps]^2_{\dot{H}^s(\RR^N)}\lesssim_{p,q} \|I_h u_\eps-u_\eps\|^{2(1-s)}_{L^q(\RR^N)} \|D(I_h u_\eps-u_\eps)\|^{2s}_{L^p(\RR^N)},
\end{equation*}
where $1<q,p<\infty$ are such that $\frac{1-s}{q}+\frac{s}{p}=\frac{1}{2}$.

We fix $q$ and $p$ such that $q,p>\frac{2N}{N+2(2-s)}$ (for example, one can choose $q=p=2$). Employing Lemmas \ref{lem:interpolation-error-Lq-norm} and \ref{lem:interpolation-error-Lp-norm} we obtain:
\begin{equation}
\label{eq:s-norm}
\int_{\RR^N} \int_{\RR^N} \frac{|(I_h u_\eps-u_\eps)(x)-(I_h u_\eps-u_\eps)(y)|^2}{|x-y|^{N+2s}} \, \dd x\, \dd y\leq C_{N,s,p,q} h^{2(2-s)(1-\gamma)} |\log h|^{2\alpha},
\end{equation}
where $C_{N,s,p,q}$ is a constant which depends only on the dimension $N$, the fractional parameter $s$ and the interpolation parameters $q$ and $p$. Thus, we get:
\begin{equation}
\label{eq:interpolant-error}
\int_{\BB} \frac{|\Pi_h u_\eps(x)|^2}{|x|^{2s}} \, \dd x  \geq  \frac{1}{1+\delta} \int_{\BB} \frac{|u_\eps(x)|^2}{|x|^{2s}} \, \dd x- \frac{1}{\delta\lambda_{N,s}} C_{N,s,p,q} h^{2(2-s)(1-\gamma)} |\log h|^{2\alpha}.
\end{equation}
Plugging estimates \eqref{eq:s-norm} and \eqref{eq:interpolant-error} into \eqref{eq:lambdah-1} and using Lemma \ref{lem:minimizing-sequence}, we get:
\begin{equation*}
\begin{aligned}
\lambda_{N,s,h}(\BB) \leq \frac{\lambda_{N,s}+C_{N,s} |\log h|^{-2} }{1-\frac{\delta}{1+\delta} -\frac{1}{\delta\lambda_{N,s}} \tilde{C}_{N,s,p,q} h^{2(2-s)(1-\gamma)} |\log h|^{-1}},
\end{aligned}
\end{equation*}
where $C_{N,s}, \tilde{C}_{N,s,p,q}$ are positive constants. Taking $\delta=\sqrt{\frac{\tilde{C}_{N,s,p,q} h^{2(2-s)(1-\gamma)} |\log h|^{-1}}{\lambda_{N,s}}}>0$ we obtain that:
\begin{equation*}
\frac{\delta}{1+\delta} +\frac{1}{\delta\lambda_{N,s}}\tilde{C}_{N,s,p,q} h^{2(2-s)(1-\gamma)} |\log h|^{-1}\leq 2\sqrt{\frac{\tilde{C}_{N,s,p,q}}{\lambda_{N,s}}} h^{(2-s)(1-\gamma)} |\log h|^{-\frac{1}{2}}<<1, \quad\text{for}\ h<<1,
\end{equation*}
and then
\begin{equation*}
\begin{aligned}
\lambda_{N,s,h}(\BB)& \leq \frac{\lambda_{N,s}+C_{N,s} |\log h|^{-2}}{1-2\sqrt{\frac{\tilde{C}_{N,s,p,q}}{\lambda_{N,s}}} h^{(2-s)(1-\gamma)} |\log h|^{-\frac{1}{2}}}\\
&\leq (\lambda_{N,s}+C_{N,s} |\log h|^{-2}) \left(1+4\sqrt{\frac{\tilde{C}_{N,s,p,q}}{\lambda_{N,s}}} h^{(2-s)(1-\gamma)} |\log h|^{-\frac{1}{2}}\right)\\
&\leq \lambda_{N,s}+3C_{N,s} |\log h|^{-2}.
\end{aligned}
\end{equation*}
The proof of the upper bound is now completed.

\section{Future research directions}
\label{sec:future-research}
The method developed here — a logarithmic improvement of the underlying inequality
combined with a ground-state representation and an explicit near-optimal competitor —
is not special to the $L^2$ fractional Hardy inequality, and it is natural to ask how far
it extends. We collect below the main $L^p$-type variants of the fractional Hardy
inequality that are not addressed in this paper, and indicate, in each case, what
additional ingredient the present approach would require.
\begin{enumerate}
\item \textbf{The $L^p$ fractional Hardy inequality of Herbst.}
Herbst \cite{Herbst1977} proved that
\begin{equation*}
\|(-\Delta)^{\frac{s}{2}} u\|^p_{L^p(\RR^N)}\geq  {\lambda}_{N,s}(p) \int_{\RR^N} \frac{|u(x)|^p}{|x|^{ps}} \, \dd x,
\end{equation*}
for some optimal constant $ {\lambda}_{N,s}(p)$, where $s>0$, $1<p<\frac{N}{s}$. Our strategy requires two things that are not yet available for
$p\ne2$: a \emph{logarithmic-improvement} inequality analogous to
Proposition~\ref{prop:log-improvement}, quantifying the deficit of $\tilde\lambda_{N,s,p}$
on bounded domains, and a \emph{remainder-term} (ground-state) representation of the
Hardy deficit playing the role of \eqref{eq:representation-for-ueps}, on which the
construction of an explicit near-optimal sequence $u_\eps$ could be based. Neither seems
to be recorded in the literature for general $p$.
\medskip

\item \textbf{Lower bounds for the Hardy deficit, and the missing upper bound.}
Frank and Seiringer \cite{FRANK20083407} proved a \emph{lower} bound for the Hardy deficit,
\begin{align*}
\int_{\RR^N}\int_{\RR^N} \frac{|u(x)-u(y)|^p}{|x-y|^{N+ps}} \, \dd x\, \dd y&-\lambda_{N,s}(p) \int_{\RR^N} \frac{|u(x)|^p}{|x|^{sp}} \, \dd x\\
&\geq c_p\int_{\RR^N}\int_{\RR^N} \frac{|v(x)-v(y)|^p}{|x-y|^{N+ps}} \frac{\, \dd x\, \dd y}{|x|^{\frac{N-ps}{2}} |y|^{\frac{N-ps}{2}}},
\end{align*}
with    $0<c_p\leq 1$, $v(x)=|x|^{\frac{N-ps}{p}}u(x)$, and with a suitable restriction on $N,s,$ and $p$.  A
comparable lower bound, under the normalization $\int_{\RR^N} |u(x)|^p|x|^{-sp}\, \dd x=1$,
was established more recently in \cite{banerjee2026quantitative}:
\begin{equation*}
    \int_{\RR^N}\int_{\RR^N} \frac{|u(x)-u(y)|^p}{|x-y|^{N+ps}} \, \dd x\, \dd y-\lambda_{N,s}(p) \gtrsim (d_{s,p}(u,\mathcal{Z}))^{\max\{4,2p\}},
\end{equation*}
$\mathcal{Z}$ denoting the family of pseudo-minimizers, with distance measured in
Marcinkiewicz space. A discretization result in the style of Theorem~\ref{thm:main} needs
a \emph{matching upper bound} on the same deficit; for $p=2$ this is exactly what the
ground-state representation \eqref{eq:representation-for-ueps} supplies. Whether an
analogous representation, together with a suitably adapted competitor $u_\eps$, can
produce the matching upper bound for $p\ne2$ — and whether the resulting discrete rate
remains logarithmic, as it does here, or degrades because of the $\max\{4,2p\}$-power
distance appearing above — is open.

\medskip
\item \textbf{The Bregman-form fractional Hardy inequality.}
Bogdan et al. \cite{krzysztof2022} proved a
Bregman-divergence version of the fractional Hardy inequality: denoting the Sobolev– Bregman form by
\[\mathcal E_p[u]=\frac12\iint(u(x)-u(y))\big(|u(x)|^{p-1}\mathrm{sgn}(u(x))-|u(y)|^{p-1}\mathrm{sgn}(u(y))\big)\nu(x,y)\,\dd y\,\dd x\]
  and
$F_p(a,b)=|b|^p-|a|^p-p|a|^{p-1}\mathrm{sgn}(a)(b-a)$ the Bregman divergence
they obtained 
\begin{equation*}
    \mathcal{E}_p[u]-k_{\frac{N-2s}{p}} \int_{\RR^N} \frac{|u(x)|^p}{|x|^{2s}} \, \dd x= \frac{1}{p} \int_{\RR^N}\int_{\RR^N} F_p\left(\frac{u(x)}{h(x)},\frac{u(y)}{h(y)}\right) h(x)^{p-1} h(y) \nu(x,y) \, \dd y \, \dd x,
\end{equation*}
where
$h(x)=|x|^{-\frac{N-2s}{p}}$, $k_{\frac{N-2s}{p}}$ an explicit optimal constant, and
$$\nu(x,y)=\frac{2^{2s}\Gamma(\frac{N+2s}{2})}{\pi^{N/2}|\Gamma(-s)|}|x-y|^{-N-2s}.$$   To our
knowledge, no logarithmic improvement of this inequality on bounded domains is currently
available in the literature; obtaining one — the direct analogue of
Proposition~\ref{prop:log-improvement} for the Bregman form — would be the first step
toward a discrete convergence result of the type established here, and is, as far as we
know, entirely open.
\end{enumerate}

Beyond enlarging the class of inequalities treated, one may also ask whether the
\emph{approximation space} $V_h$ itself can be replaced. A recent direction in this
spirit, motivated by neural PDE solvers, replaces the piecewise linear finite element
space by a realization class $\mathcal{M}_p$ of linear combinations of Gaussian functions
with freely varying centers; for a model elliptic problem, Zuazua \cite{zuazua2026} shows
that a logarithmic convergence rate persists for this nonlinear approximation class as
well. It would be of interest to determine whether the near-optimal sequence $u_\eps$
constructed in Section~\ref{sec:upper-bound} — or a Gaussian-mixture analogue of it — can
be projected onto $\mathcal{M}_p$ with a comparable logarithmic loss, which would suggest
that the $|\log h|^{-2}$ rate in Theorem~\ref{thm:main} is intrinsic to the fractional
Hardy deficit itself, rather than an artifact of the piecewise linear discretization.

\subsection*{Acknowledgements}
 A. Dima was partially supported by a scholarship of the SCOSAAR. 
L. I.  Ignat  was partially supported by a grant of the Ministry of Research, Innovation, and Digitization, CCCDI -
UEFISCDI, project number ROSUA-2024-0001, within PNCDI IV. The authors wish to thank  Krzysztof Bogdan for fruitful discussions. This article is based upon work from COST Action 24122 mSPACE, supported by COST (European Cooperation in Science and Technology), www.cost.eu.

\appendix

\section{Technical Lemmas}
\begin{lemma}
\label{lem:hardy-type-inequality}
Let $\alpha,\beta,r>0$ and $s\in \left(0,\min\left\{1,\frac{N}{2}\right\}\right)$, $s\neq \frac{1}{2}$. The following inequality holds for any function $u\in H^1(B_r(0))$:
\begin{equation*}
(N-2s-r^{2-2s} \alpha)\int_{B_r(0)} \frac{u^2(x)}{|x|^{2s}} \, \dd x\leq (N+\beta) r^{-2s} \int_{B_r(0)} u^2(x) \, \dd x+\left(\frac{r^{2-2s}}{\beta}+\frac{1}{\alpha}\right) \int_{B_r(0)} |Du(x)|^2 \, \dd x.
\end{equation*}
\end{lemma}
\begin{proof}
The proof is a slight modification of the one from \cite[Theorem 7, p. 296]{evans2010partial}. It is enough to consider that $u\in C^\infty(B_r(0))$. Since we assume that $s\neq \frac{1}{2}$, we have:
\begin{equation*}
\begin{aligned}
\int_{B_r(0)} \frac{u^2(x)}{|x|^{2s}} \, \dd x&=\int_{B_r(0)} u^2(x) D\left(\frac{|x|^{1-2s}}{1-2s}\right) \frac{x}{|x|} \, \dd x\\
&=-\frac{1}{1-2s} \int_{B_r(0)} 2u(x)Du(x) \frac{x}{|x|^{2s}} \, \dd x-\frac{N-1}{1-2s} \int_{B_r(0)} \frac{u^2(x)}{|x|^{2s}} \, \dd x\\
&+\frac{1}{1-2s} \int_{\partial B_r(0)} u^2(x) \frac{x}{|x|} \frac{x}{|x|^{2s}} \, \dd S,
\end{aligned}
\end{equation*}
so that 
\begin{equation}
\label{eq:hardy-type-inequality-main}
(N-2s) \int_{B_r(0)} \frac{u^2(x)}{|x|^{2s}} \, \dd x= r^{1-2s}\int_{\partial B_r(0)} u^2(x) \, \dd S-\int_{B_r(0)} 2u(x)Du(x) \frac{x}{|x|^{2s}} \, \dd x.
\end{equation}
Using that $|2ab|\leq \alpha a^2 + \frac{1}{\alpha} b^2$ for any $a,b\in \RR$, we get that
\begin{equation}
\begin{aligned}
\label{eq:hardy-type-inequality-alpha}
\left|\int_{B_r(0)}  2u(x)Du(x) \frac{x}{|x|^{2s}} \, \dd x\right| & \leq \int_{B_r(0)} 2|Du(x)| \frac{|u(x)|}{|x|^{2s-1}} \, \dd x\leq \alpha \int_{B_r(0)} \frac{u^2(x)}{|x|^{4s-2}} \, \dd x+\frac{1}{\alpha} \int_{B_r(0)} |Du(x)|^2 \, \dd x\\
&=\alpha\int_{B_r(0)} \frac{u^2(x)}{|x|^{2s}} |x|^{2-2s} \, \dd x+\frac{1}{\alpha} \int_{B_r(0)} |Du(x)|^2 \, \dd x\\
&\leq \alpha r^{2-2s} \int_{B_r(0)} \frac{u^2(x)}{|x|^{2s}} \, \dd x+ \frac{1}{\alpha} \int_{B_r(0)} |Du(x)|^2 \, \dd x.
\end{aligned}
\end{equation}
Next,
\begin{equation}
\label{eq:integral-on-partial}
r\int_{\partial B_r(0)} u^2(x) \, \dd S=\int_{B(0,r)} div(x u^2(x)) \, \dd x=N\int_{B_r(0)} u^2(x) \, \dd x + \int_{B_r(0)} 2u(x) Du(x) x \, \dd x.
\end{equation}
Using the same inequality as before, we get that
\begin{equation*}
\begin{aligned}
\left|\int_{B_r(0)} 2u(x)Du(x) x \, \dd x\right|&\leq \beta\int_{B_r(0)} u^2(x) \, \dd x+ \frac{1}{\beta}\int_{B_r(0)} |Du(x)|^2 |x|^2 \, \dd x\\
&\leq \beta\int_{B_r(0)} u^2(x) \, \dd x+ \frac{r^2}{\beta} \int_{B_r(0)}|Du(x)|^2 \, \dd x.
\end{aligned}
\end{equation*}
Thus,
\begin{equation}
\label{eq:hardy-type-inequality-beta}
r^{1-2s}\int_{\partial B_r(0)} u^2(x) \, \dd S\leq (N+\beta) r^{-2s} \int_{B_r(0)} u^2(x)\, \dd x+\frac{r^{2-2s}}{\beta} \int_{B_r(0)} |Du(x)|^2 \, \dd x.
\end{equation}
The result now follows from \eqref{eq:hardy-type-inequality-main}, \eqref{eq:hardy-type-inequality-alpha} and \eqref{eq:hardy-type-inequality-beta}.
\end{proof}

\begin{lemma}
\label{lem:hardy-type-inequality-2}
Let $N\geq 2$ and $r>0$. The following inequality holds for any function $u\in H^1(B_r(0))$:
\begin{equation}
\label{eq:hardy-type-inequality-2-main}
(N-1) \int_{B_r(0)} \frac{u^2(x)}{|x|} \, \dd x\leq (N+2) r^{-1} \int_{B_r(0)} u^2(x) \, \dd x+2r\int_{B_r(0)} |Du(x)|^2 \, \dd x. 
\end{equation}
\end{lemma}
\begin{proof}
We use the same method as in Lemma \ref{lem:hardy-type-inequality}. We can consider that $u\in C^{\infty}(B_r(0))$. Since $$D\left(\frac{1}{|x|\log\left(\frac{er}{|x|}\right)}\right)=\frac{x}{|x|^3 \log^2\left(\frac{er}{|x|}\right)}-\frac{x}{|x|^3 \log\left(\frac{er}{|x|}\right)},$$ we get that $$\frac{1}{|x|}=\frac{1}{|x|\log\left(\frac{er}{|x|}\right)}-D\left(\frac{1}{|x|\log\left(\frac{er}{|x|}\right)}\right) x\log\left(\frac{er}{|x|}\right).$$ Consequently,
\begin{equation*}
\begin{aligned}
\int_{B_r(0)} \frac{u^2(x)}{|x|} \, \dd x&=\int_{B_r(0)} \frac{u^2(x)}{|x|\log\left(\frac{er}{|x|}\right)} \, \dd x-\int_{B_r(0)} u^2(x) D\left(\frac{1}{|x|\log\left(\frac{er}{|x|}\right)}\right) x\log\left(\frac{er}{|x|}\right) \, \dd x\\
&=\int_{B_r(0)} \frac{u^2(x)}{|x| \log\left(\frac{er}{|x|}\right)} \, \dd x-\int_{\partial B_r(0)} u^2(x) \, \dd S\\
&+\int_{B_r(0)} 2u(x) Du(x) \frac{x}{|x|} \, \dd x+\int_{B_r(0)} \frac{u^2(x)}{|x|\log\left(\frac{er}{|x|}\right)} div\left(x\log\left(\frac{er}{|x|}\right)\right)\, \dd x\\
&=\int_{B_r(0)} \frac{u^2(x)}{|x| \log\left(\frac{er}{|x|}\right)} \, \dd x-Nr^{-1}\int_{B_r(0)} u^2(x) \, \dd x\\
&-r^{-1} \int_{B_r(0)} 2u(x) Du(x) x \, \dd x+\int_{B_r(0)} 2u(x) Du(x) \frac{x}{|x|} \, \dd x\\
&+N\int_{B_r(0)} \frac{u^2(x)}{|x|} \, \dd x-\int_{B_r(0)} \frac{u^2(x)}{|x|\log\left(\frac{er}{|x|}\right)} \, \dd x,
\end{aligned}
\end{equation*}
where we have used \eqref{eq:integral-on-partial} and that $div\left(x\log\left(\frac{er}{|x|}\right)\right)=N\log\left(\frac{er}{|x|}\right)-1$. Thus, we obtain
\begin{equation*}
\begin{aligned}
(N-1)\int_{B_r(0)} \frac{u^2(x)}{|x|} \, \dd x&=N r^{-1}\int_{B_r(0)} u^2(x) \, \dd x +r^{-1} \int_{B_r(0)} 2u(x) Du(x) x \, \dd x\\
&-\int_{B_r(0)} 2u(x) Du(x) \frac{x}{|x|} \, \dd x.
\end{aligned}
\end{equation*}
Using the following two inequalities
\begin{equation*}
\left|\int_{B_r(0)} 2u(x) Du(x) x \, \dd x\right|\leq \int_{B_r(0)} u^2(x) \, \dd x + r^2 \int_{B_r(0)} |Du(x)|^2 \, \dd x,
\end{equation*}
and
\begin{equation*}
\begin{aligned}
\left|\int_{B_r(0)} 2u(x) Du(x) \frac{x}{|x|} \, \dd x\right|&\leq \int_{B_r(0)} 2\left|\frac{u(x)}{\sqrt{r}} Du(x) \frac{x\sqrt{r}}{|x|}\right| \, \dd x\\
&\leq r^{-1} \int_{B_r(0)} u^2(x) \, \dd x+r\int_{B_r(0)} |Du(x)|^2 \, \dd x,  
\end{aligned}
\end{equation*}
we obtain the desired conclusion.
\end{proof}

\begin{lemma}
\label{lem:D2ueps}
The gradient and the Hessian of the minimizing sequence $u_\eps$ satisfy:
\begin{equation*}
|D u_{\eps}(x)|\lesssim |x|^{-\frac{N+2(1-s)}{2}} \log^{\alpha}\left(\frac{1}{|x|}\right),  \quad \text{if}\ |x|\in \left(\eps^2,\frac{1}{2}\right),
\end{equation*}
and
\begin{equation*}
|D^2 u_\eps(x)|\lesssim |x|^{-\frac{N+2(2-s)}{2}} \log^{\alpha}\left(\frac{1}{|x|}\right), \quad \text{if}\ |x|\in \left(\eps^2,\frac{1}{2}\right)
\end{equation*}
and they both vanish otherwise.
\end{lemma}
\begin{proof}
Using that, for a $C^2(\RR^N)$ radial function $v=v(|x|)=v(r)$, $\partial_{x_i x_j}v(x)=v''(r)\frac{X_i X_j}{r^2}+v'(r)\frac{\delta_{i j} r^2-X_i X_j}{r^3}$ we get that $$|D^2 u_\eps(x)|^2=(u''_\eps(r))^2+(N-1)\left(\frac{u'_\eps(r)}{r}\right)^2, \quad\text{for any}\  x\in\RR^N.$$\\
If $|x|\leq\eps^2$ or $|x|\geq \frac{1}{2}$ the conclusion easily follows. If $|x|\in \left(\eps^2,\frac{1}{2}\right)$, then, since $$\|\psi\|_{L^{\infty}(\RR)}, \|\psi'\|_{L^{\infty}(\RR)}, \|\psi''\|_{L^{\infty}(\RR)} \lesssim 1,$$ we can suppose, without loss of generality that $\psi(|x|)=1$ if $|x|\in \left(\frac{1}{4},\frac{1}{2}\right)$. Using the properties of $\eta_\eps$ we get that:
\begin{equation*}
u_\eps'(r)=-\frac{N-2s}{2} r^{-\frac{N+2(1-s)}{2}} \log^{\alpha}\left(\frac{1}{r}\right) \eta_\eps(r)+r^{-\frac{N-2s}{2}} \alpha \log^{\alpha-1}\left(\frac{1}{r}\right) \left(-\frac{r}{r^2}\right) \eta_\eps(r)\\
+r^{-\frac{N-2s}{2}} \log^{\alpha}\left(\frac{1}{r}\right) \eta'_\eps(r)
\end{equation*}
so that $$\frac{|u'_\eps(r)|}{r}\lesssim r^{-\frac{N+2(2-s)}{2}} \log^{\alpha}\left(\frac{1}{r}\right)$$
and 
\begin{equation*}
\begin{split}
u_\eps''(r)&=\frac{(N-2s)(N+2(1-s))}{4} r^{-\frac{N+2(2-s)}{2}} \log^{\alpha}\left(\frac{1}{r}\right) \eta_\eps(r)-\frac{N-2s}{2} r^{-\frac{N+2(1-s)}{2}}\alpha \log^{\alpha-1}\left(\frac{1}{r}\right) \left(-\frac{r}{r^2}\right) \eta_\eps(r)\\
&-\frac{N-2s}{2} r^{-\frac{N+2(1-s)}{2}} \log^{\alpha}\left(\frac{1}{r}\right) \eta'_\eps(r)+\frac{N+2(1-s)}{2} r^{-\frac{N+2(2-s)}{2}} \alpha\log^{\alpha-1}\left(\frac{1}{r}\right) \eta_\eps(r)\\
&-r^{-\frac{N+2(1-s)}{2}} \alpha(\alpha-1) \log^{\alpha-2}\left(\frac{1}{r}\right) \left(-\frac{r}{r^2}\right) \eta_\eps(r)-r^{-\frac{N+2(1-s)}{2}} \alpha \log^{\alpha-1}\left(\frac{1}{r}\right) \eta'_\eps(r)\\
&-\frac{N-2s}{2} r^{-\frac{N+2(1-s)}{2}} \log^{\alpha}\left(\frac{1}{r}\right) \eta'_\eps(r) +r^{-\frac{N-2s}{2}} \alpha \log^{\alpha-1}\left(\frac{1}{r}\right) \left(-\frac{r}{r^2}\right) \eta_\eps'(r)\\
&+r^{-\frac{N-2s}{2}} \log^{\alpha}\left(\frac{1}{r}\right) \eta''_\eps(r)\\
\end{split}
\end{equation*}
so that $$|u''_\eps(r)|\lesssim r^{-\frac{N+2(2-s)}{2}} \log^{\alpha}\left(\frac{1}{r}\right).$$
The estimates in the Lemma now follow.
\end{proof}
\begin{lemma}
\label{lem:interpolation-error-Lq-norm}
Let $\eps=h^{\frac{\gamma}{2}}$ for some $\gamma\in (0,1)$. Then for any $\max\left\{1,\frac{2N}{N+2(2-s)}\right\}<q<\infty$ we get the following estimate for the $L^q(\RR^N)-$norm of the interpolation error:
\begin{equation*}
\|u_\eps-I_h u_\eps\|_{L^q(\RR^N)}\lesssim h^{\gamma\frac{N}{q}+2-\gamma\frac{N+2(2-s)}{2}} |\log h|^{\alpha}.
\end{equation*}
\end{lemma}
\begin{proof}
Since $u_\eps$ vanishes inside $B_{\eps^2}(0)$ and outside $B_{\frac{1}{2}}(0)$ we get that $$\|u_\eps-I_h u_\eps\|_{L^q(\RR^N)}=\|u_\eps-I_h u_\eps\|_{L^q(B_{\frac{1}{2}}(0)\setminus {B}_{h^\gamma}(0))}.$$
We define the following set of mesh triangles: $$\TT^1_h\coloneqq \{T\in \TT_h: T\cap B_{\frac{h^{\gamma}}{2}}(0)=\emptyset, T\subset B_{\frac{1}{2}+h}(0)\}.$$
Thus, $$\|u_\eps-I_h u_\eps\|^q_{L^q(\RR^N)}\leq\sum_{T\in \TT^1_h} \int_{T} |(u_\eps-I_h u_\eps)(x)|^q \, \dd x.$$
Using \cite[Lemma 3.1]{Ignat2025} we get that
\begin{equation}
\label{eq:interpolation-error-Lq-norm-1}
\|u_\eps-I_h u_\eps\|^q_{L^q(\RR^N)}\lesssim_q h^{2q} \sum_{T\in \TT^1_h} |T| \|D^2 u_\eps\|^q_{L^{\infty}(T)}.
\end{equation}
Using Lemma \ref{lem:D2ueps}, we obtain 
\begin{equation}
\label{eq:interpolation-error-Lq-norm-2}
|T| \|D^2 u_\eps\|^q_{L^{\infty}(T)}\lesssim |T| \sup_{x\in T}  |x|^{-q\frac{N+2(2-s)}{2}} \log^{q\alpha}\left(\frac{1}{|x|}\right).
\end{equation}
We aim to return to the $L^q(T)$ setting by proving that the norm of points in a triangle $T\in \TT^1_h$ are comparable with each other. Indeed, for every $x,y\in T\in \TT^1_h$ we have $$|x|\geq \frac{h^{\gamma}}{2}\geq \frac{9h}{2}\geq \frac{9h_T}{2}\geq \frac{9|x-y|}{2}\geq \frac{9|x|}{2}-\frac{9|y|}{2},$$
and, as a result, $9|y|\geq 7|x|$. It follows that: $$\sup_{y\in T} |y|\leq \frac{9}{7}\inf_{y\in T} |y|,$$
which further implies that 
\begin{equation*}
\begin{aligned}
\sup_{y\in T} |y|^{-q\frac{N+2(2-s)}{2}} \log^{q\alpha}\left(\frac{1}{|y|}\right)&\leq \inf_{y\in T} |y|^{-q\frac{N+2(2-s)}{2}} \left(\frac{9}{7}\right)^{q\frac{N+2(2-s)}{2}} \log^{q\alpha}\left(\frac{9}{7|y|}\right)\\
&\leq C_{q,N,s} \inf_{y\in T} |y|^{-\frac{N+2(2-s)}{2}} \log^{q\alpha}\left(\frac{1}{|y|}\right),
\end{aligned}
\end{equation*}
where $C_{q,N,s}$ is a constant which depends only on the exponent $q$, the dimension $N$ and the fractional parameter $s$. Plugging this inequality into \eqref{eq:interpolation-error-Lq-norm-1} and \eqref{eq:interpolation-error-Lq-norm-2} we obtain that:
\begin{equation*}
\begin{split}
\|u_\eps-I_h u_\eps\|^q_{L^q(\RR^N)} &\lesssim_q h^{2q} \sum_{T\in \TT^1_h} \int_{T}  |x|^{-q\frac{N+2(2-s)}{2}} \log^{q\alpha}\left(\frac{1}{|x|}\right) \, \dd x\\
& \lesssim_q h^{2q} \int_{B_{\frac{1}{2}+h}(0)\setminus B_{\frac{h^\gamma}{2}}(0)} |x|^{-q\frac{N+2(2-s)}{2}} \log^{q\alpha}\left(\frac{1}{|x|}\right) \, \dd x\\
&\lesssim_q h^{2q} \int_{h^\gamma}^{\frac{1}{2}+h} r^{N-1-q\frac{N+2(2-s)}{2}} \log^{q\alpha}\left(\frac{1}{r}\right) \, \dd r\\
& \lesssim_q h^{2q} |\log h|^{q\alpha} \int_{h^\gamma}^{\frac{1}{2}+h} r^{N-1-q\frac{N+2(2-s)}{2}} \, \dd r\\
&\lesssim_q h^{2q +\gamma N-\gamma q\frac{N+2(2-s)}{2}} |\log h|^{q\alpha}.
\end{split} 
\end{equation*}
\end{proof}
\begin{lemma}
\label{lem:interpolation-error-Lp-norm}
Let $\eps=h^{\frac{\gamma}{2}}$ for some $\gamma\in (0,1)$. Then for any $\max\left\{1,\frac{2N}{N+2(2-s)}\right\}<p<\infty$ we get the following estimate for the $\dot{W}^{1,p}(\RR^N)-$norm of the interpolation error:
\begin{equation*}
\|D(u_\eps-I_h u_\eps)\|_{L^p(\RR^N)}\lesssim_p h^{\gamma\frac{N}{p}+1-\gamma\frac{N+2(2-s)}{2}} |\log h|^{\alpha}.
\end{equation*}
\end{lemma}
\begin{proof}
We proceed as in the proof of Lemma \ref{lem:interpolation-error-Lq-norm}. Using \cite[Lemma 3.1]{Ignat2025} we get that:
\begin{equation*}
\begin{split}
\|D(u_\eps-I_h u_\eps)\|^p_{L^p(\RR^N)} &\lesssim_p h^{p} \sum_{T\in \TT^1_h} \int_{T}  |x|^{-p\frac{N+2(2-s)}{2}} \log^{p\alpha}\left(\frac{1}{|x|}\right) \, \dd x\\
& \lesssim_p h^{p} \int_{B_{\frac{1}{2}+h}(0)\setminus B_{\frac{h^\gamma}{2}}(0)} |x|^{-p\frac{N+2(2-s)}{2}} \log^{p\alpha}\left(\frac{1}{|x|}\right) \, \dd x\\
&\lesssim_p h^{p} \int_{h^\gamma}^{\frac{1}{2}+h} r^{N-1-p\frac{N+2(2-s)}{2}} \log^{p\alpha}\left(\frac{1}{r}\right) \, \dd r\\
& \lesssim_p h^{p} |\log h|^{p\alpha} \int_{h^\gamma}^{\frac{1}{2}+h} r^{N-1-p\frac{N+2(2-s)}{2}} \, \dd r\\
&\lesssim_p h^{p +\gamma N-\gamma p\frac{N+2(2-s)}{2}} |\log h|^{p\alpha}.
\end{split} 
\end{equation*}
\end{proof}
\printbibliography

@misc{zuazua2026,
      title={The Coercivity Gap in Neural PDE Solvers: Parameter Escape and Functional Convergence}, 
      author={Enrique Zuazua},
      year={2026},
      eprint={2606.04018},
      archivePrefix={arXiv},
      primaryClass={math.NA},
      url={https://arxiv.org/abs/2606.04018}, 
}

@article{pratelli,
author = {Antonietti, Paola and Pratelli, Aldo},
year = {2011},
month = {12},
pages = {37-64},
title = {Finite element approximation of the Sobolev constant},
volume = {117},
journal = {Numerische Mathematik},
doi = {10.1007/s00211-010-0347-y}
}

@article{Ignat2025,
title = {Sharp numerical approximation of the Hardy constant},
journal = {Discrete and Continuous Dynamical Systems},
year = {2025},
issn = {1078-0947},
doi = {10.3934/dcds.2025165},
url = {https://www.aimsciences.org/article/id/68f09b4fcb5dde21e755067d},
author = {Liviu I. Ignat and Enrique Zuazua},
}

@incollection{Tzirakis23,
author = {Konstantinos Tzirakis},
title = {Stability Estimates for Fractional Hardy-Schrödinger Operators},
booktitle = {Fixed Point Theory and Chaos},
publisher = {IntechOpen},
address = {London},
year = {2023},
editor = {Guillermo Huerta-Cuellar},
chapter = {6},
doi = {10.5772/intechopen.109606},
url = {https://doi.org/10.5772/intechopen.109606}
}

@article{FRANK20083407,
title = {Non-linear ground state representations and sharp Hardy inequalities},
journal = {Journal of Functional Analysis},
volume = {255},
number = {12},
pages = {3407-3430},
year = {2008},
issn = {0022-1236},
doi = {https://doi.org/10.1016/j.jfa.2008.05.015},
url = {https://www.sciencedirect.com/science/article/pii/S0022123608002140},
author = {Rupert L. Frank and Robert Seiringer},
}

@book{GLeoni2023,
    AUTHOR = {Giovanni Leoni},
    TITLE = {A First Course in Fractional Sobolev Spaces} ,
    PUBLISHER = {American Mathematical Society},
    YEAR = {2023},
}

@book{evans2010partial,
  title={Partial differential equations},
  author={Evans, Lawrence C},
  volume={19},
  year={2010},
  publisher={American Mathematical Society},
  address={Providence, RI},
  edition={2nd},
  isbn={978-0-8218-4974-3}
}

@book{BrennerScott,
    AUTHOR = {Susanne C. Brenner and L. Ridgway Scott},
    TITLE = {The Mathemathical Theory of Finite Element Methods},
    PUBLISHER = {Springer},
    YEAR = {2008}
}

@article{borthagaray2018,
author = {Borthagaray, Juan and Del Pezzo, Leandro and Martinez, Sandra},
year = {2018},
month = {10},
pages = {},
title = {Finite Element Approximation for the Fractional Eigenvalue Problem},
volume = {77},
journal = {Journal of Scientific Computing},
doi = {10.1007/s10915-018-0710-1}
}

@article{chenweth2021,
author = {Huyuan, Chen and Weth, Tobias},
year = {2021},
month = {03},
title = {The Poisson problem for the fractional Hardy operator: Distributional identities and singular solutions},
journal = {Transactions of the American Mathematical Society},
doi = {10.1090/tran/8443}
}

@article{ignat2024,
author = {Della Pietra, Francesco and Fantuzzi, Giovanni and Ignat, Liviu I. and Masiello, Alba Lia and Paoli, Gloria and Zuazua Iriondo, Enrique},
 faupublication = {yes},
 journal = {Journal of Convex Analysis},
 pages = {497-523},
 title = {{Finite} {Element} {Approximation} of the {Hardy} {Constant}},
 volume = {31},
 year = {2024}
}

@article{Herbst1977,
  title = {Spectral theory of the operator $(p^{2}+m^{2})^{1/2} - Ze^{2}/r$},
  author = {Herbst, I. W.},
  journal = {Communications in Mathematical Physics},
  volume = {53},
  pages = {285--294},
  year = {1977},
  publisher = {Springer},
  doi = {10.1007/BF01609852}
}

@article {frankJAMS2008,
    AUTHOR = {Frank, Rupert L. and Lieb, Elliott H. and Seiringer, Robert},
     TITLE = {Hardy-{L}ieb-{T}hirring inequalities for fractional
              {S}chr\"odinger operators},
   JOURNAL = {J. Amer. Math. Soc.},
  FJOURNAL = {Journal of the American Mathematical Society},
    VOLUME = {21},
      YEAR = {2008},
    NUMBER = {4},
     PAGES = {925--950},
      ISSN = {0894-0347,1088-6834},
   MRCLASS = {35P20 (26D15 35J10 35P15 46E35 47F05 81Q10)},
  MRNUMBER = {2425175},
MRREVIEWER = {G\"unter\ Berger},
       DOI = {10.1090/S0894-0347-07-00582-6},
       URL = {https://doi.org/10.1090/S0894-0347-07-00582-6},
}

@article{dima2026galerkin,
  title={Galerkin Approximation of the Fractional Sobolev Constant},
  author={Dima, Andreea and Ignat, Liviu I},
  journal={arXiv preprint arXiv:2605.13347},
  year={2026}
}

@article{banerjee2026quantitative,
  title={Quantitative stability for fractional Hardy inequalities: Rearrangement-free techniques and Emden-Fowler analysis},
  author={Banerjee, Avas and Ganguly, Debdip and Sahu, Vivek},
  journal={arXiv preprint arXiv:2605.15748},
  year={2026}
}

@article{krzysztof2022,
title = {Optimal Hardy inequality for the fractional Laplacian on $L^p$},
journal = {Journal of Functional Analysis},
volume = {282},
number = {8},
year = {2022},
issn = {0022-1236},
doi = {https://doi.org/10.1016/j.jfa.2022.109395},
url = {https://www.sciencedirect.com/science/article/pii/S0022123622000155},
author = {Krzysztof Bogdan and Tomasz Jakubowski and Julia Lenczewska and Katarzyna Pietruska-Pałuba}
}

@article{TZIRAKIS20164513,
title = {Sharp trace Hardy–Sobolev inequalities and fractional Hardy–Sobolev inequalities},
journal = {Journal of Functional Analysis},
volume = {270},
number = {12},
pages = {4513-4539},
year = {2016},
issn = {0022-1236},
doi = {https://doi.org/10.1016/j.jfa.2015.11.016},
url = {https://www.sciencedirect.com/science/article/pii/S0022123615004814},
author = {Konstantinos Tzirakis},
}

@article{ignat2025optimalconvergenceratesfinite,
  title={Optimal convergence rates for the finite element approximation of the Sobolev constant},
  author={Ignat, Liviu I and Zuazua, Enrique},
  journal={arXiv preprint arXiv:2504.09637},
  year={2025}
}
\end{document}